\documentclass[10pt,conference]{IEEEtran}
\IEEEoverridecommandlockouts
\usepackage{graphicx}
\usepackage{amssymb,amsmath,amsthm}
\usepackage{url} 
\usepackage{float}
\usepackage{subcaption}
\usepackage{algorithm}
\usepackage{algpseudocode}
\usepackage{color}
\usepackage{xcolor} %Can remove. Only for comments. 
\usepackage{url}
\usepackage{tikz} %For mesh diagram
\usepackage[normalem]{ulem}

\DeclareMathOperator*{\argmin}{arg\,min}
\algdef{SE}[SUBALG]{Indent}{EndIndent}{}{\algorithmicend\ }%
\algtext*{Indent}
\algtext*{EndIndent}

\newcommand{\commentout}[1]{\unskip}

\newcommand{\ignore}[1]{}

\newtheorem{theorem}{Theorem}[section]

\title{Two-Stage Block Orthogonalization to\\ Improve Performance of $s$-step GMRES}
\author{
\IEEEauthorblockN{
Ichitaro Yamazaki\IEEEauthorrefmark{1},
Andrew J. Higgins\IEEEauthorrefmark{2},
Erik G. Boman\IEEEauthorrefmark{1},
Daniel B. Szyld\IEEEauthorrefmark{2}
}

\IEEEauthorblockA{
\IEEEauthorrefmark{1}Sandia National Laboratories, Albuquerque, New Mexico, U.S.A\\
\IEEEauthorrefmark{2}Temple University, Philadelphia, Pennsylvania, U.S.A 
}
}
\begin{document}

\maketitle

\begin{abstract}

Generalized Minimum Residual (GMRES) is a popular Krylov
subspace projection method
for solving a large nonsymmetric linear system of equations. It computes the approximate
solution with the minimum residual norm in the generated Krylov projection subspace.
On current computer architectures,
the solver performance can be limited by its communication cost to generate
the orthonormal basis vectors of the projection subspace.
To address this potential performance bottleneck,
its $s$-step variant orthogonalizes a block of $s$ basis vectors at a time,
providing the potential to reduce the communication cost by a factor of $s$.
Unfortunately, for a large step size $s$, the solver can generate extremely ill-conditioned basis vectors, and
to maintain the stability in practice, a conservatively small step size is used,
which limits the performance advantages that the $s$-step solver can provide.
To enhance the solver performance using a small step size,
in this paper, we introduce a two-stage block orthogonalization scheme.
Similar to the original scheme,
the first stage of the proposed method operates on the block of $s$ basis vectors at a time, but
its objective is to maintain the well-conditioning of the generated basis vectors
with a lower cost.
The orthogonalization of the basis vectors is delayed until the second stage
when enough basis vectors are generated to obtain higher performance.
Our analysis shows the stability of the proposed two-stage scheme.
At the same time,
the performance is improved because while still requiring the same amount of computation as the original scheme,
the proposed scheme performs the majority of the communication at the second stage, reducing the overall communication requirements.
Our performance results with up to 192 NVIDIA V100 GPUs on the Summit supercomputer demonstrate that when solving a 2D Laplace problem,
the two-stage approach can reduce the orthogonalization time and the total time-to-solution by the respective factors of up to $2.6\times$ and $1.6\times$
over the original $s$-step GMRES, which had already obtained the respective speedups of $2.1\times$ and $1.8\times$ over the standard GMRES.
Similar speedups were obtained for 3D problems and for matrices from the SuiteSparse Matrix Collection.
%
%The random sketching also alleviates the need of tuning the step size 
%since the performance improvement can be obtained using a relatively small step size.
\end{abstract}

\section{Introduction}

Generalized Minimum Residual (GMRES)~\cite{Saad:1986} is a popular Krylov subspace projection method
for iteratively solving a large nonsymmetric linear system of equations.
At each iteration, it generates a new Krylov basis vector for the projection subspace 
using a sparse-matrix vector multiply (SpMV), typically combined with a preconditioner
to accelerate its solution convergence rate.
The basis vector is then orthonormalized   
to maintain the numerical stability 
of generating the projection subspace in finite precision
and to compute the approximate solution that minimizes the $\ell_2$ residual norm in the projection subspace.
As the subspace dimension grows,
it becomes expensive to generate the orthonormal basis vectors
in terms of both computation and storage.
To reduce the costs of computing a large subspace,
the iteration is restarted after a fixed number $m+1$ of basis vectors are computed.
%
%GMRES is popular in practice due to its optimal approximate solution convergence
%with the minimal residual norm.

To orthogonalize the new basis vector at each iteration, GMRES uses BLAS-1 and BLAS-2 operations,
which have limited potential for data reuse, and requires the global reduces among all the MPI processes.
On current computers, these communication 
(e.g., the cost of moving data through the local memory hierarchy and between the MPI processes)
can take much longer than the required computation time
and can limit the performance of the orthogonalization process.
As a result, when efficient and scalable SpMV and preconditioners are available, 
orthogonalization could become a significant part
of the iteration time and a performance bottleneck.

To reduce this potential performance bottleneck,
communication-avoiding (CA) variants of GMRES~\cite{Carson:2015,Hoemmen:2010}, based on $s$-step methods~\cite{Sturler:1995,Joubert:1992}, 
were proposed.
To generate the orthogonal basis vectors of the Krylov projection subspace,
the $s$-step GMRES utilizes two computational kernels:
1) the Matrix Powers Kernel (MPK) to generate the $s+1$ Krylov vectors by applying SpMV and preconditioner $s$ times, followed by
2) the Block Orthogonalization Kernel that orthogonalizes a block of $s+1$ basis vectors at once.
Since the block orthogonalization kernel performs most of its local computation using BLAS-3
and synchronizes only every $s$ steps, compared to the standard GMRES,
the $s$-step variant has the potential to reduce the communication cost
of orthogonalizing the $s$ basis vectors by a factor of~$s$. 
This is a very attractive feature, especially on the currently available GPU clusters, 
where the communication can be significantly more expensive compared to computation.

Unfortunately, though mathematically equivalent, for a larger step size,
MPK can generate extremely ill-conditioned $s$-step basis vectors.
Hence, in practice, in order to maintain the stability of MPK, 
a conservatively small step size is used,
which limits the performance advantages of $s$-step GMRES over the standard GMRES.
In this paper, we introduce a two-stage orthogonalization scheme
to improve the performance of the $s$-step GMRES, while still 
using a small step size $s$ to maintain the stability of MPK.

There are two main contributions in this paper. First, we analyze
the current state-of-the-art block orthogonalization algorithms
for $s$-step GMRES (Section~\ref{sec:bcgs2}). This motivates a new combination of the block orthogonalization algorithms, which we call BCGS-PIP2.
Though this new variant improves the performance of the original algorithms, it still has two synchronizations every $s$ steps.
Second, to further enhance the performance,
we propose and extend the study to the two-stage approach, which delays one of the synchronizations until enough number of basis vectors, $\widehat{s}$, are generated to obtain higher performance (Section~\ref{sec:two-stage}). 
In other words, though the two-stage approach performs about the same amount of computation as the original algorithms,
it performs half of the local computation using the larger block size $\widehat{s}$ instead of the original step size $s$, hence increasing the potential for the data reuse.
In addition, the two-stage approach performs only one synchronization at the first stage (every $s$ steps),
while delaying the other synchronization until the second stage (every $\widehat{s}$ steps).
In particular, if we set the second step size same as the Krylov subspace projection dimension (i.e., $\widehat{s}=m$), the two-stage approach provides the potential to reduce the communication cost
by a factor of two.

We demonstrate the potential of the new variant and of the two-stage approach
through numerical and performance experiments (Sections~\ref{sec:numerics} and \ref{sec:performance}):
\begin{itemize}
    \item We study the numerical stability of the new variant and that of the two-stage approach. We clarify the conditions that each of the algorithms requires to maintain its stability, and present numerical experiments to demonstrate the
          numerical properties of the algorithms.

    \item We implement the two-stage approach in Trilinos~\cite{trilinos-website}, which is a collection of
          open-source software packages for developing
          large-scale scientific and engineering simulation codes. Trilinos software stack
          allows the solvers, like $s$-step GMRES, to be portable to different computer architectures,
          using a single code base.
          
    \item We present GPU performance of $s$-step GMRES, combined with the two-stage approach.
          Our performance results on the Summit supercomputer demonstrate that when solving a 2D Laplace problem on 192 NVIDIA V100 GPUs,
          our two-stage approach can obtain speedups of $2.6\times$ and $1.6\times$ 
          for orthogonalization and for the total time-to-solution, respectively, over the original $s$-step GMRES,
          which had already obtained the respective speedups of $2.1\times$ and $1.8\times$ over the standard GMRES.
          Similar speedups were observed for 3D model problems and for matrices from the SuiteSparse Matrix Collection.
\end{itemize}
The two-stage approach also alleviates the need of fine-tuning the step size for each problem on a specific hardware
since a conservatively small step-size may be used for numerical stability
while relying on the two-stage approach to obtain the performance improvement.

Table~\ref{tab:notation} lists the notation used in this paper. In addition, we use $Q_{\ell:t}$ to denote the blocks column vectors of $Q$ with the block column indexes $\ell$ to $t$, while $q_{k:s}$ is the set of vectors with the column indexes $k$ to $s$. Finally,
$[Q,V]$ is the column concatenation of $Q$ and $V$.

% ----------------------------------------------------
\begin{table}
\centerline{\footnotesize
\begin{tabular}{c|l}
notation    & description\\
\hline
%\hline
 $n$        & problem size\\
 %\hline
 $m$        & subspace dimension \\
 %\hline
 $s$        & step size (for the first stage)\\
 %\hline
 $\widehat{s}$   & second step size (for the second stage and $s \le \widehat{s} \le m$)\\
 %\hline
 $v_k^{(j)}$ & $k$th basis vector within $s$ basis vectors\\
 %\hline
$V_j$       & $j$th $s$-step basis vectors including the starting vector, i.e.,\\
            & a set of $s+1$ vectors generated by MPK \\
            & $V_j = [v_{s(j-1)+1},v_{s(j-1)+2},\dots,v_{sj+1}]$\\
            & and $V_0 = [v_0]$ to simplify the notation\\
 %\hline
$\underline{V}_j$
            & same as $V_j$ except excluding the last vector, \\
            & which is the first vector of $V_{j+1}$, i.e.,\\
            & a set of $s$ vectors $\underline{V}_j = [v_{s(j-1)+1},v_{s(j-1)+2},\dots,v_{sj}]$\\
 %\hline
$\widehat{V}_j$ & $V_j$ after the first inter-block orthogonalization\\
$\widehat{Q}_j$ & $V_j$ after the pre-processing stage\\
$Q_j$           & orthogonal basis vectors of $V_j$\\
 %\hline
$\epsilon$  & machine epsilon\\
 %\hline
$\kappa(V_j)$ & condition number of $V_j$
\end{tabular}}
\caption{Notation used in the paper.} \label{tab:notation}
\end{table}
% ----------------------------------------------------

\section{Related Work}

%To generate the orthogonal basis vectors of the projection subspace,
%$s$-step GMRES relies on the Matrix-Powers Kernel (MPK),
%which generates a set of $s$ new Krylov vectors, 
%and the block orthogonalization kernel.
%There are several techniques to reduce the communication cost of
%applying a sequence of SpMVs,
%which is often referred to as CA variants of the MPK~\cite{Mohiyuddin:2009}.
%There are also techniques to improve the numerical stability of MPK~\cite{Bai:1994}.
%Instead, our focus of the paper is on improving the performance of the block orthogonalization process.

The block orthogonalization is a critical component in many applications including linear or eigen solvers, and is an active research area. 
There are several combinations of the block orthogonalization schemes~\cite{Carson:2022}, but, especially in terms of the performance on current computer architectures, the Block Classical Gram-Schmidt (BCGS), combined with some variants of Cholesky QR (CholQR)~\cite{Stath:2002}, is considered the state-of-the-art. This paper builds and extends on this combination. Some techniques that are relevant to this paper include:
\begin{itemize}
\item
CholQR computes the QR factorization of a tall and skinny matrix.
Unfortunately, CholQR can fail when the condition number of the input matrix is greater than
the reciprocal of the square-root of the machine precision
(it computes the Cholesky factorization of the Gram matrix of the input basis vectors to be orthogonalized,
and the Gram matrix has the condition number which is the square of the input vectors' condition number).
Nonetheless, it performs well on current computer architectures because most of its local computation is based on BLAS-3 and it requires just one global reduce.
Hence, it is still used in practice but requires some remedies to maintain its stability.
In addition, to maintain the orthogonality of the column vectors,
it is often applied with reorthogonalization (referred to as CholQR twice, or equivalently CholQR2 for short).

\item Shifted Cholesky QR~\cite{Fukaya:2020} is introduced to avoid this numerical instability of CholQR.
Though it may require one additional round of the orthogonalization, increasing
     the computational and communication costs of CholQR2 by a factor of $1.5\times$,
      it has the stability guarantee as long as the input vectors are numerically full-rank.

\item
A mixed-precision variant of CholQR~\cite{Yamazaki:2014:mpChol}, which has similar stability properties as the shifted CholQR, was proposed. To ensure stability, the Gram matrix is accumulated in double the working precision. When working in double precision, it requires quadruple precision, which can be software-emulated through double-double precision arithmetic if quadruple precision is not supported by the hardware~\cite{Hida:2001}. Though double-double arithmetic has high computational overhead compared to double precision, the mixed-precision CholQR does not increase the communication cost significantly.
When the performance of CholQR is dominated by communication, it may obtain performance similar to the standard CholQR. Its application to the block orthogonalization has also been studied~\cite{Yamazaki:2015}.

\item
There are low-synchronous variants of block orthogonalization algorithms that reduce the number of synchronizations
and improve the performance of orthogonalization~\cite{Carson:2022,Yamazaki:2020}. 
These techniques require an efficient low-synchronous intra-block orthogonalization algorithm.
Though there is a CA tall and skinny QR factorization algorithm that is unconditionally stable~\cite{Demmel:2012},
its local computation is based on Householder QR (HHQR) factorization,
which is mainly based on BLAS-1 or BLAS-2 and may obtain much lower performance than BLAS-3 based CholQR.
Hence, in practice,
these low-synchronous techniques rely on some variant of CholQR factorization for orthogonalizing each block. 
\end{itemize}
Though some of the techniques mentioned above have improved stability, the $s$-step basis vectors, generated by MPK, can be extremely ill-conditioned for a large step size $s$, and in order to ensure the stability in practice, $s$-step GMRES still needs to use a small step-size. Since the performance of the orthogonalization may be limited by the multiple synchronizations required at every $s$ steps, in this paper, we look at avoiding or delaying some of the synchronizations, while using a small step size $s$ to maintain stability. Moreover, the proposed two-stage approach may be combined with these previous approaches. In particular,
random-sketching techniques have been recently integrated into CholQR~\cite{Balabanov:2022}. We are investigating the potential of randomized CholQR to improve the stability of our block orthogonalization process.

\section{$s$-step GMRES}
\label{sec:sstep}

% ----------------------------------------------------
\begin{figure}[t]
\begin{center}
  \centerline{\fbox{\begin{minipage}[h!]{.45\textwidth}
    \footnotesize
    \input{codes/gmres}
  \end{minipage}}}
\end{center}
  \caption{Pseudocode of $s$-step GMRES where $[Q_j,R_j] = \mbox{BlkOrth}(Q,V_j)$ returns the QR factorization such that $Q R = V$ 
           with $Q^TQ = I$ and $R$ is upper triangular with non-negative diagonals. }\label{algo:sstep}
\end{figure}
% ----------------------------------------------------

Fig.~\ref{algo:sstep} shows the pseudocode of $s$-step GMRES
for solving a linear system $Ax=b$,
which has been also implemented in the Trilinos software framework. 

Compared to the standard GMRES, this $s$-step variant has the potential of reducing the communication cost 
of generating the $s$ orthonormal basis vectors by a factor of $s$,
where the standard GMRES is essentially $s$-step GMRES with the step size of one.
For instance, to apply SpMV $s$ times (Lines 7 to 9 of the pseudocode),
several CA variants of the ``matrix-powers kernel'' (MPK) have been proposed~\cite{Mohiyuddin:2009}.
However, to reduce the communication latency on a distributed-memory computer,
CA MPK requires additional memory and local computation, and may also increase the total communication volume.
More critically, in practice, SpMV is typically applied together with a preconditioner to accelerate the convergence rate of GMRES.
Although a few CA preconditioners of specific types have been proposed \cite{Grigori:2015,Yamazaki:2014}, avoiding communication 
for other types of preconditioners is still an open research problem. 
To support a wide range of application needs, instead of CA MPK,
Trilinos $s$-step GMRES uses a standard MPK (applying each SpMV with neighborhood communication and preconditioner in sequence),
and focuses
on improving the performance of block orthogonalization by reducing its communication costs.
Also, avoiding the global communication in orthogonalization could lead to a greater performance gain than CA MPK does, when scalable implementations of SpMV and preconditioner are available. 
This motivates our study of the block orthogonalization in this paper.

\section{Block Orthogonalization}
\label{sec:bcgs2}

% ----------------------------------------------------
\begin{figure}[t]
\begin{center}
  \begin{subfigure}[b]{\linewidth}
  \centerline{\fbox{\begin{minipage}[h!]{.9\linewidth}
    \footnotesize
    \input{codes/BCGS}
  \end{minipage}}}
  \caption{Block Classical Gram-Schmidt (BCGS) for inter-block Orthogonalization.}
  \end{subfigure}
  \begin{subfigure}[b]{\linewidth}
  \centerline{\fbox{\begin{minipage}[h!]{.9\linewidth}
    \footnotesize
    \input{codes/BCGS2-v2}
  \end{minipage}}}
  \caption{Block Classical Gram-Schmidt twice (BCGS2) for inter-block orthogonalization, combined with HHQR or CholQR2 intra-block orthogonalization.} \label{algo:bcgs2}
  \end{subfigure}
\end{center}
  \caption{ Block Classical Gram-Schmidt to generate a new set of orthonormal basis vectors $Q_j$. HHQR$(\widehat{V}_j)$ returns the QR factorization of $\widehat{V}_j$
           based on the Householder algorithm, while the pseudocode of CholQR2 is shown in Fig.~\ref{algo:cholqr2}. } \label{algo:bcgs}
\end{figure}

\begin{figure}
\begin{center}
  \begin{subfigure}[b]{\linewidth}
  \centerline{\fbox{\begin{minipage}[h!]{.9\linewidth}
    \footnotesize
    \input{codes/cholQR_v1}
  \end{minipage}}}
  \caption{Cholesky QR (CholQR).}
  \end{subfigure}
  \begin{subfigure}[b]{\linewidth}
  \centerline{\fbox{\begin{minipage}[h!]{.9\linewidth}
    \footnotesize
    \input{codes/cholQR2}
  \end{minipage}}}
  \caption{Cholesky QR twice (CholQR2).} \label{algo:cholqr2}
  \end{subfigure}
\end{center}
  \caption{
           Intra-block Cholesky QR to orthonormalize a set of vectors $\widehat{V} \in \mathbb{R}^{n\times s+1}$, where $\mbox{Chol}(G)$ returns the upper-triangular Cholesky factor of the Gram matrix $G$.} \label{algo:cholqr}
\end{figure}
% ----------------------------------------------------

The block orthogonalization algorithm in $s$-step GMRES consists of two algorithms:
the \emph{inter} and \emph{intra} block orthogonalization to orthogonalize the new block of $s+1$ basis vectors against the already-orthogonalized previous blocks of vectors and 
among the vectors within the new block, respectively.
There are several combinations of the inter- and intra-block orthogonalization algorithms~\cite{Carson:2022},
but the state-of-the-art inter-block algorithm is based on the Block Classical Gram-Schmidt (BCGS), 
which is entirely based on BLAS-3 operations and requires only one global reduce.
As a result, BCGS obtains superior performance on current computers. 

To maintain orthogonality, in practice, BCGS is applied
with re-orthogonalization (BCGS twice, or BCGS2). 
Fig.~\ref{algo:bcgs} shows pseudocode of BCGS2,
which has two algorithmic options for the first intra-block orthogonalization, while CholQR is used for the second intra-block orthogonalization. 
As we discuss in more detail below, with these combinations of the inter and intra block-orthogonalization algorithms, the orthogonality errors of the computed basis vectors $Q_j$ can be bounded by $\mathcal{O}(\epsilon)$, where $\epsilon$ is the machine precision.
%which is mainly based on BLAS-3 operations,
%and requires four global reduces.
For our discussion of BCGS2 using different algorithms such as HHQR or CholQR2 for the first intra-block orthogonalization, we refer them as ``BCGS2 with HHQR'' or ``BCGS2 with CholQR2'', respectively.

\subsection{BCGS2 with HHQR}

When the column vectors of the input matrix $V$ are numerically full-rank (i.e., $\kappa(V) \max\{n,s\}~\epsilon < 1$),
BCGS2 with HHQR in Fig.~\ref{algo:bcgs2} generates the orthonormal basis vectors~$Q$ with orthogonality error on the order of machine precision, i.e., $\|I-Q^TQ\| = \mathcal{O}(\epsilon)$~\cite{Barlow:2021,Barlow:2013}.
Unfortunately, for the small step size that we typically use (e.g., $s=5$ is the default step size in Trilinos),
the HHQR of $\widehat{V}_j$ is based on BLAS-1 or BLAS-2 and requires $\mathcal{O}(s)$ global reduces,
which often lead to the performance of HHQR and overall BCGS2, which is much lower
than the peak performance of the current computers (e.g., based on the memory bandwidth).

There have been significant advances in the theoretical understanding of $s$-step Krylov methods~\cite{Carson:2015}.
However, though the orthogonality error bound to obtain the backward stability of GMRES has been established~\cite{Greenbaum:1997},
to the authors knowledge,
there are no known theoretical bounds on the orthogonality errors, which are required to obtain the maximum attainable accuracy of $s$-step GMRES.
Hence, in this paper, we focus on the block orthogonalization schemes that can maintain the $\mathcal{O}(\epsilon)$ orthogonality error, like BCGS2 with HHQR does
(though this might not be needed to obtain the maximum accuracy of $s$-step GMRES), while improving the performance of the block orthogonalization.

\subsection{BCGS2 with CholQR2}

To generate the orthonormal basis vectors of $\widehat{V}_j$, HHQR and CholQR2 require about the same amount of computation. However,
as the pseudocode in Fig.~\ref{algo:cholqr} shows, in contrast to HHQR, 
CholQR is mainly based on BLAS-3 and requires only one synchronization.
As a result, on current computer architectures, CholQR
often obtains much higher performance than HHQR,
and BCGS2 with CholQR2 is considered to be one of the state-of-the-art block-orthogonalization algorithms in terms of performance.
Hence, in this paper,
we focus on BCGS2 with CholQR2 and discuss when it obtains the same stability as BCGS2 with HHQR.

In~\cite{Stath:2002, Fukaya:2015}, it was shown that when
the condition number of the input vectors $\widehat{V}_j$ is bounded as
\begin{equation}\label{eq:assumption-1}
  c_1(\epsilon, n,s)\kappa(\widehat{V}_j)^2 < 1/2,
\end{equation}
the orthogonality error of $\widetilde{V}_j$ computed by the first CholQR on Line~2 of Fig.~\ref{algo:cholqr2} is bounded by
\begin{equation}\label{eq:cholqr_bound}
  %\|I - Q_j^TQ_j\| \le \mathcal{O}(\epsilon)\kappa(V_j)^2.
  \|I - \widetilde{V}_j^T\widetilde{V}_j\| \le c_1(\epsilon, n,s)\kappa(\widehat{V}_j)^2,
\end{equation}
where the scalar term $c_1(\epsilon, n,s)$ is
\begin{equation}\label{eq:cholqr_constant}
    c_1(\epsilon, n,s) = 5\left( ns+s(s+1)\right)\epsilon.
\end{equation}
Condition~\eqref{eq:assumption-1} implies that the Cholesky factorization of the Gram matrix of $\widehat{V}_j$ is numerically stable,
and also that all the Krylov basis vectors generated by MPK are numerically full-rank (otherwise GMRES has converged). 

When condition \eqref{eq:assumption-1}, and hence the orthogonality error bound~\eqref{eq:cholqr_bound},
hold, we have the following theorem showing that CholQR2 is as stable as HHQR.
\begin{theorem}\label{thm:chol}
With the bound~\eqref{eq:cholqr_bound} and assumption~\eqref{eq:assumption-1}, 
the condition number of $\widetilde{V}_j$ computed by the first CholQR
(on Line 2 in Fig.~\ref{algo:cholqr2}) is bounded by
\[
  \kappa(\widetilde{V}_j) < \sqrt{3}.
\]
and hence, the orthogonality error of $\widehat{Q}_j$ computed by CholQR2 satisfies
\[
  \|I - \widehat{Q}_j^T \widehat{Q}_j\| = \mathcal{O}(\epsilon).
\]
\begin{proof}
Let $\sigma_1(G) \ge \dots \ge \sigma_s(G)$ be the singular values of the Gram matrix of $\widetilde{V}_j$, i.e., $G = \widetilde{V}_j^T\widetilde{V}_j$, and hence $\sigma_k(G) = \sigma_k(\widetilde{V}_j)^2$ for $k=1,\dots,s$. Then, with the  upper-bound~\eqref{eq:cholqr_bound} and assumption~\eqref{eq:assumption-1}, along with Weyl's inequality \cite{Weyl:1912}, we have
\[
\left\{
\begin{array}{l} %lclclcl}
  \sigma_1(\widetilde{V}_j)^2  \le 1 + \|\widetilde{V}_j^T\widetilde{V}_j - I\| \le 1 +  c_1(\epsilon, n,s)\kappa(\widehat{V}_j)^2
             <  3/2\\
  \sigma_s(\widetilde{V}_j)^2  \ge 1 - \|I - \widetilde{V}_j^T\widetilde{V}_j\| \ge 1 -  c_1(\epsilon, n,s)\kappa(\widehat{V}_j)^2
             >  1/2
\end{array}
\right.
\]
giving the above upper-bound on the condition number of $\widetilde{V}_j$
and the orthogonality error of $\widehat{Q}_j$.
\end{proof}
\end{theorem}
\noindent
Hence, overall, when condition \eqref{eq:assumption-1} is satisfied,
BCGS2 with CholQR2 generates the basis vectors $Q$ with the orthogonality error of the order of the machine precision.

% ----------------------------------------------------
\begin{figure}[t]
\begin{center}
  \begin{subfigure}[b]{.9\linewidth}
  \centerline{\fbox{\begin{minipage}[h!]{\textwidth}
    \footnotesize
    \input{codes/BCGS-PIP}
  \end{minipage}}}
  \caption{BCGS with Pythagorean Inner Product (BCGS-PIP). } \label{algo:bcgs-pip}
  \end{subfigure} 
  \begin{subfigure}[b]{\linewidth}
  \centerline{\fbox{\begin{minipage}[h!]{.9\linewidth}
    \footnotesize
    \input{codes/BCGS2-pip}
  \end{minipage}}}
  \caption{BCGS-PIP twice (BCGS-PIP2).}\label{algo:bcgs-pip2}
  \end{subfigure} 
\end{center}
  \caption{ BCGS with Pythagorean Inner Product to generate a new set of orthonormal basis vectors $Q_j$. } \label{algo:bcgs2-pip}
\end{figure}
% ----------------------------------------------------

\subsection{BCGS-PIP2}

Recently, a ``single-reduce'' variant of BCGS with CholQR based on Pythagorean Inner Product (BCGS-PIP) was proposed~\cite{Carson:2022,Yamazaki:2020}. The pseudocode of BCGS-PIP is shown in Fig.~\ref{algo:bcgs-pip},
which orthogonalizes a new set of basis vectors $V_j$ against the previously-orthonormalized basis vectors $Q_{1:j-1}$.
Instead of explicitly computing the Gram matrix of $\widehat{V}_j$ for CholQR, BCGS-PIP computes it by updating the Gram matrix of $V_j$
based on the block generalization of the  Pythagorean theorem, allowing to orthonormalize the block vector $V_j$ with a single all-reduce. %When there is no previous vectors (i.e., $j=1$), BCGS-PIP becomes CholQR.

Moreover, it was shown~\cite[Theorem 3.4]{Carson:2022} that 
if the previous basis vectors have been orthogonalized to satisfy\footnote{%
This is a much stronger condition than that required by \cite[Theorem 3.4]{Carson:2022},
but is satisfied by the algorithm discussed in this paper.}
\begin{eqnarray}
  && \label{eq:assumption-Q}\|I - Q_{1:j-1}^T Q_{1:j-1}\| = \mathcal{O}(\epsilon),
\end{eqnarray}
and the MPK generates the next set of the block vector $V_j$ such that
\begin{eqnarray}
  \label{eq:assumption-2}
  && c_2(\epsilon)\kappa([Q_{1:j-1},V_j])^2 < 1/2,
\end{eqnarray}
then the orthogonality error of $\widehat{Q}_j$ computed by BCGS-PIP satisfies
\small\begin{equation}
 \|I - [Q_{1:j-1},\widehat{Q}_j]^T[Q_{1:j-1},\widehat{Q}_j]\| \le c_3(\epsilon)\kappa([Q_{1:j-1},V_j])^2, \label{eq:orthBoundQhatQ}
\end{equation}\normalsize
where $c_2(\epsilon)$ and~$c_3(\epsilon)$ are two functions that behave asymptotically like $\mathcal{O}(\epsilon)$,
similar to that given by~\eqref{eq:cholqr_constant}.
% when the MPK generates each block $V_j$ of vectors for $j=1,\dots,m/s$ such that
% \begin{eqnarray}
%   \label{eq:assumption-2}
%   && f(\epsilon)\kappa([Q_{1:j-1},V_j])^2 < 1/2\\
%   \label{eq:assumption-Q}
%   && \mbox{ with } \|I - Q_{1:j-1}^T Q_{1:j-1}\| = \mathcal{O}(\epsilon),
% \end{eqnarray}
% where the assumption~\eqref{eq:assumption-Q} is satisfied due to Theorem~\ref{thm:chol}.
% \ichi{
% In other words, with the condition~\eqref{eq:assumption-2} satisfied, and with $\widehat{Q}_j$ computed by BCGS-PIP,
% we have a $\mathcal{O}(1)$ condition number of the basis vectors $[Q_{1:j-1},\widehat{Q}_j]$, and we can obtain $Q_{1:j}$ with $\mathca{O}(\epsilon)$ orthogonality error by calling BCGS-PIP on $[Q_{1:j-1},\widehat{Q}_j]$.
% The pseudoccode of the resulting BCGS-PIP twice (BCGS-PIP2) is shown in Fig.~\ref{algo:bcgs2-pip}.
% }

Here, in order to generate the orthogonal basis vectors~$Q_j$
with the orthogonality error of the order of the machine precision,
we apply BCGS-PIP twice (BCGS-PIP2).
The pseudocode of the resulting algorithm is shown in Fig.~\ref{algo:bcgs2-pip}.
If conditions~\eqref{eq:assumption-Q} and \eqref{eq:assumption-2} are satisfied, we have the following theorem showing the stability of 
BCGS-PIP2.
\begin{theorem}\label{thm:bcgs-pip2}
    With the bound~\eqref{eq:orthBoundQhatQ} and the assumptions~\eqref{eq:assumption-Q} and \eqref{eq:assumption-2},
    BCGS-PIP computes $\widehat{Q}_j$ such that the condition number of the accumulated basis vectors
$[Q_{1:j-1},\widehat{Q}_j]$ satisfies,
    \begin{equation}
    \kappa([Q_{1:j-1},\widehat{Q}_j]) = \mathcal{O}(1), \label{eq:hatQcondBound}
    \end{equation}
    and 
    the resulting $Q_j$ from BCGS-PIP2$(Q_{1:j-1},V_j)$ satisfies 
    \begin{eqnarray}
        \label{eq:orthBoundBCGS2-pip}
      && \|I - Q_{1:j}^T Q_{1:j}\| = \mathcal{O}(\epsilon).
    \end{eqnarray}
\end{theorem}

\begin{proof}
The proof is based on Weyl's inequality similar to that for Theorem~\ref{thm:chol}.
\end{proof}
While BCGS2 with CholQR2 needs five synchronizations every $s$ steps, this new variant BCGS-PIP2 needs just two synchronizations and reduces the total computational cost of the intra-block orthogonalization by a factor of $1.5\times$.

We note that when there are no previous blocks (i.e., $j=1$),
BCGS-PIP2 is CholQR2, which satisfies the condition~\eqref{eq:assumption-Q}
due to Theorem~\ref{thm:chol}. Theorems~\ref{thm:chol} and \ref{thm:bcgs-pip2}
imply that when the required assumptions hold,
BCGS followed by CholQR and BCGS-PIP
are both stable pre-processing algorithms for BCGS2 with CholQR2 and BCGS-PIP2, respectively, and obtain $\mathcal{O}(1)$ condition number of the pre-processed block vectors.

% ----------------------------------------------------
%\begin{figure}[t]
%  \centerline{\fbox{\begin{minipage}[h!]{\linewidth}
%    \footnotesize
%    \input{codes/blockQR}
%  \end{minipage}}}
%  \caption{\label{algo:block-ortho}Block Orthogonalization process;
%           BCGS2 followed by CholQR2 to generate the $\ell_2$ orthonormal basis vector of the Krylov subspace $\mathchal{K}_m(A,v_0)$.}
%\end{figure}
%% ----------------------------------------------------

Nevertheless, to obtain the best performance of the block orthogonalization, the step size ${s}$ needs to be carefully chosen for each problem on a different hardware. Unfortunately, it is often infeasible to fine-tune the step size in practice, while even if Newton or Chebyshev basis~\cite{Bai:1994} are used, for a large step size $s$, MPK can generate ill-conditioned basis vectors $V_j$ with a large condition number. Hence, instead of fine-tuning the step size, in practice, a conservatively small step size is used
to satisfy the conditions discussed in this section and avoid the numerical instability (e.g., $s=5$). Though BCGS-PIP2 improves the orthogonalization performance by reducing the number of the synchronizations, the small step size may still limit the performance gain that $s$-step methods can bring.

\section{Two-stage Block Orthogonalization}
\label{sec:two-stage}

% ----------------------------------------------------
\begin{figure}[t]
\begin{center}
  \centerline{\fbox{\begin{minipage}[h!]{.45\textwidth}
    \footnotesize
    \input{codes/BCGS2-2stage}
  \end{minipage}}}
\end{center}
  \caption{
           Pseudocode of two-stage algorithms to generate the orthonormal basis vectors of
           ``Big Panel'' consisting of $\widehat{s}+1$ Krylov vectors.} \label{algo:2step}
\end{figure}
% ----------------------------------------------------

In order to improve the performance of block orthogonalization in $s$-step GMRES while using a small step size $s$,
we propose a ``two-stage'' block orthogonalization process, shown in Fig.~\ref{algo:2step}.
Instead of performing BCGS-PIP2 at every $s$ steps to generate the orthonormal basis vectors,
we call BCGS-PIP only once on the new $s+1$ basis vectors generated by MPK. The objective of this first stage is to pre-process the $s$-step basis vectors to maintain a small condition number of the generated basis vectors. In particular, since the $s$-step basis vectors $\widehat{Q}_{1:j-1}$ are being ``roughly'' orthogonalized, when MPK generates the next $s$-step basis vector $V_j$ using the last vector of the block $\widehat{Q}_{j-1}$ as the starting vecctor, the condition number of the accumulated basis vectors $[\widehat{Q}_{\ell:j-1},V_j]$ is hoped to be roughly the same as that of $V_j$ (we will show the numerical results in Section~\ref{sec:numerics}). Then once a sufficient number of basis vectors, $\widehat{s}$, are generated to obtain higher performance,
we orthogonalize the $\widehat{s}$ basis vectors at once by calling BCGS-PIP for the second time, but now on a larger block size $\widehat{s}$ instead of the original step size $s$.

This two-stage approach in Fig.~\ref{algo:2step}
is similar to BCGS-PIP2 in Fig.~\ref{algo:bcgs2-pip}.
To compare these two approaches,
we distinguish between the blocks of two different block sizes $s$ and $\widehat{s}$ for the first and second stages
by referring to them as the ``panels'' and ``big panels'', respectively.
Hence, the two-stage approach pre-processes the panels of $s$ columns at a time, followed by BCGS-PIP on the big panel of $\widehat{s}$ columns.
For the two extreme cases, with $\widehat{s} = s$ and $\widehat{s}=m$
the two-stage approach becomes the standard one-stage BCGS-PIP2 and BCGS-PIP on each panel, followed by CholQR on the entire $m+1$ basis vectors, respectively.

Compared to the original one-stage BCGS-PIP2,
the two-stage approach performs about the same number of floating point operations, but it reduces the number of global synchronizations and performs the orthogonalization using a larger block size.
In particular,
BCGS2 with CholQR2 in Fig.~\ref{algo:bcgs2} or BCGS-PIP2 
in Fig.~\ref{algo:bcgs2-pip}
performs five or two synchronizations at every $s$ steps, respectively.
In contrast, the two-stage approach has only one synchronization every $s$ steps, and one more synchronization every $\widehat{s}$ steps.
Hence, with $\widehat{s} = m$, the two-stage approach reduces the number of synchronization by the factor of $2\times$ (and could also reduce the amount of the required data movement through the local memory hierarchy).
As a result, the two-stage approach often obtains much higher performance as we show in Section~\ref{sec:performance}.

In lieu of rigorous roundoff error analysis, we will provide intuition behind the expected orthogonality errors of the two-stage approach. Since the two-stage approach uses BCGS-PIP on each panel and then on the big panel, we can apply error analysis similar to Section~\ref{sec:bcgs2} but require the assumption~\eqref{eq:assumption-2} on the big panel.
In particular, if the following condition on the big panel $V_{\ell:t}$ is satisfied: 
\begin{eqnarray}\label{eq:assumption-3}
  && c_2(\epsilon)\kappa([Q_{1:\ell-1},V_{\ell:t}])^2 < 1/2\\ 
  \nonumber
  && \mbox{ with } \|I - Q_{1:\ell-1}^T Q_{1:\ell-1}\| = \mathcal{O}(\epsilon),
\end{eqnarray}
then by~\cite[Theorem 3.4]{Carson:2022}, the first BCGS-PIP pre-processes the big panel such that the generated basis vectors satisfy
\small\begin{equation}
 \|I - [Q_{1:\ell-1},\widehat{Q}_{\ell:t}]^T[Q_{1:\ell-1},\widehat{Q}_{\ell:t}]\|
   \le c_3(\epsilon)\kappa([Q_{1:\ell-1},V_{\ell:t}])^2. \label{eq:orthBigPanelFirstStage}
\end{equation}\normalsize
Hence, similar to Theorem~\ref{thm:bcgs-pip2}, under condition~\eqref{eq:assumption-3},
after the first stage, we expect the condition number of the accumulated big panels $[Q_{1:\ell-1},\widehat{Q}_{\ell:t}]$ to be of $\mathcal{O}(1)$, and when the big panel $\widehat{Q}_{\ell:t}$ is orthogonalized by BCGS-PIP at the second stage, 
we expect an $\mathcal{O}(\epsilon)$ orthogonality error of $Q_{1:t}$ by
combining \eqref{eq:orthBigPanelFirstStage} with the $\mathcal{O}(1)$ condition number.

Specifically, by applying Weyl's inequality, we have the following theorem.
\begin{theorem}\label{thm:two-stage}
With the condition~\eqref{eq:assumption-3} and the bound~\eqref{eq:orthBigPanelFirstStage},
the condition number of the big panel after the pre-processing is given by
\begin{equation}
 \kappa([Q_{1:\ell-1},\widehat{Q}_{\ell:t}]) = O(1). \label{eq:condBigPanelFirstStage}
\end{equation}
As a result, when the second stage calls BCGS-PIP on $[Q_{1:\ell-1},\widehat{Q}_{\ell:t}]$, the orthogonality error of the generated basis is bounded as
\begin{equation}
   \|I - Q_{1:t}^T Q_{1:t}\| = \mathcal{O}(\epsilon).
\end{equation}
\end{theorem}
The new condition~\eqref{eq:assumption-3} is on the condition number of the big panel,
while the condition~\eqref{eq:assumption-2} for BCGS-PIP2 only required the condition number of each panel to be less than $\mathcal{O}(\epsilon^{-1/2})$. The key feature of the two-stage approach is that the starting vector for MPK is the last column of $\widehat{Q}_{j-1}$, which has been pre-processed by BCGS-PIP, whose objective is to maintain the small enough condition number of the big panel that satisfies~\eqref{eq:assumption-3}. 

In the next section, we study how the condition number of the pre-processed big panel grows, and compare the orthogonality errors of the two-stage approach with the standard algorithms.
In particular, the numerical results show that 
\[
  \kappa([Q_{1:\ell-1},V_{\ell:j}]) \approx \kappa([Q_{1:\ell-1},\hat{Q}_{\ell:j-1},V_j]),
\]
making assumption \eqref{eq:assumption-3} required for the two-stage BCGS-PIP2 a similar stability requirement to assumption \eqref{eq:assumption-2} required for the one-stage BCGS-PIP2.

%\section{Random Sketching}
%\input{sketch}

%\section{Asymptotic Complexity}
%\input{complexity}

\section{Numerical Results}
\label{sec:numerics}
 
% ============================== %
\begin{figure}[t]
   \centerline{
     \includegraphics[width=.7\linewidth]{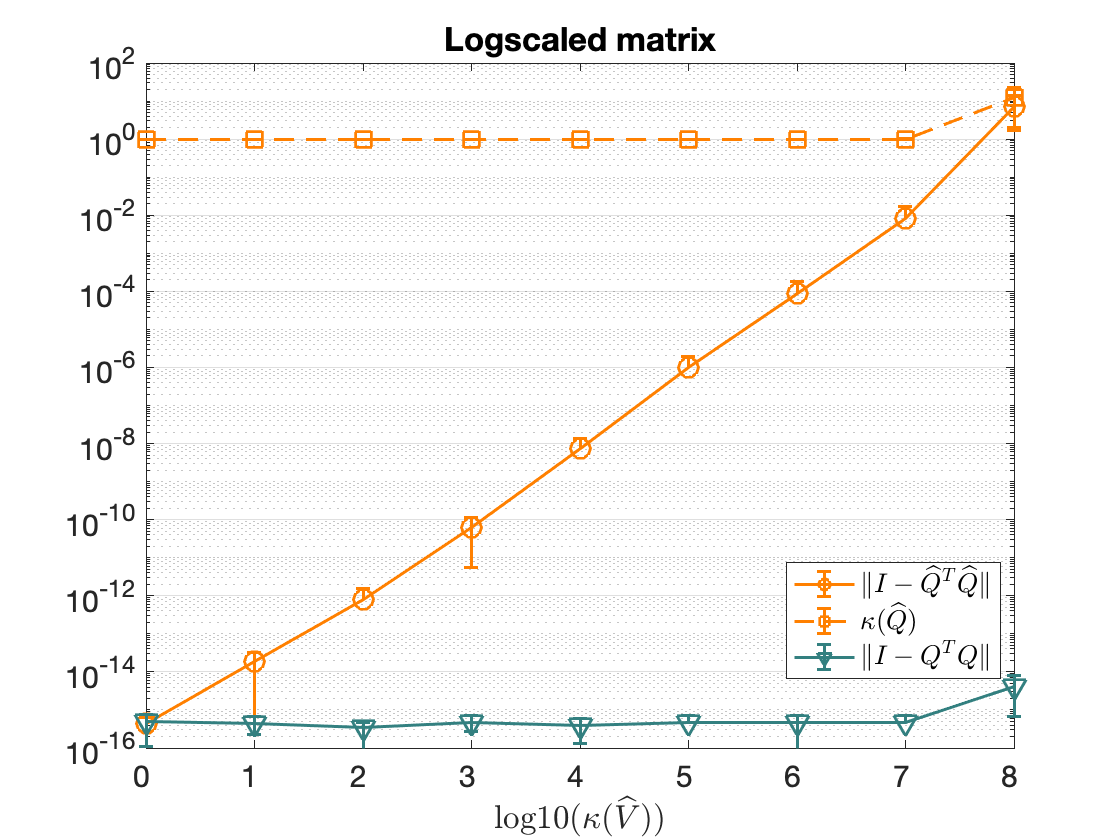}
    }
 \caption{Orthogonality error and condition number with CholQR2 on $10^5$-by-$5$ {\tt Logscaled} matrix.} \label{fig:cholqr-error}
\end{figure}
% ============================== %

We compared the orthogonality errors of the proposed block orthogonalization schemes using the default double precision in MATLAB. We first study how the orthogonality errors grow with the condition number of the input vectors. For these studies, instead of studying the numerical properties of the proposed methods within $s$-step GMRES, we treat them as a general block orthogonalization scheme and use synthetic matrices as the input vectors such that we can control the condition number of the matrix easily. We start by showing that both CholQR2 and one-stage BCGS-PIP2 obtain $\mathcal{O}(\epsilon)$ orthogonality error given that conditions~\eqref{eq:assumption-1} and~\eqref{eq:assumption-2} are satisfied, respectively (Figs.~\ref{fig:cholqr-error} and \ref{fig:error-pip2}). We then show that the orthogonality errors of the two-stage approach are also $\mathcal{O}(\epsilon)$ when the condition~\eqref{eq:assumption-3} is satisfied (Fig.~\ref{fig:error-2stage}). Since these synthetic matrices are generated using random numbers, we show the minimum, average, and maximum errors using ten different random seeds. 
Finally, we study how the condition numbers grow for the basis vectors generated by MPK using various positive indefinite matrices of dimension between 200,000 and 300,000 from the SuiteSparse Matrix Collection\footnote{\url{https://sparse.tamu.edu}} (Fig.\ref{fig:cond-mpk}). In order to maintain the stability of the original $s$-step method, we scaled the columns and then rows of the matrices by the maximum nonzero entries in the columns and rows (hence, all the resulting matrices are non-symmetric).
For all of our experiments with MPK and $s$-step GMRES, we used monomial basis, even though using more stable bases, like Newton or Chebyshev bases, could reduce the condition number and improve the applicability of our approaches to a wider class of problems.

% ============================== %
\begin{figure}[t]
\centerline{
  \begin{subfigure}[b]{.53\linewidth}
   \centerline{
     \includegraphics[width=\linewidth]{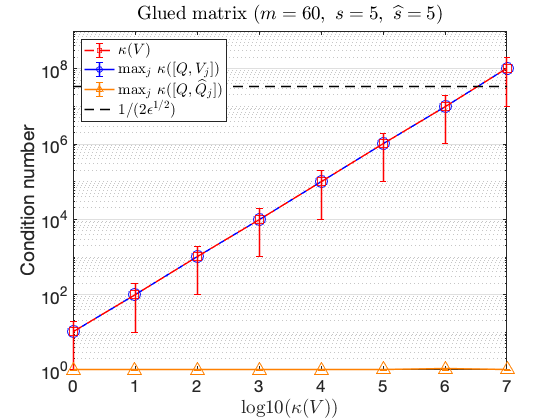}
    }
    \caption{Condition number.}\label{fig:cond-pip2}
  \end{subfigure}
  \begin{subfigure}[b]{.53\linewidth}
   \centerline{
     \includegraphics[width=\linewidth]{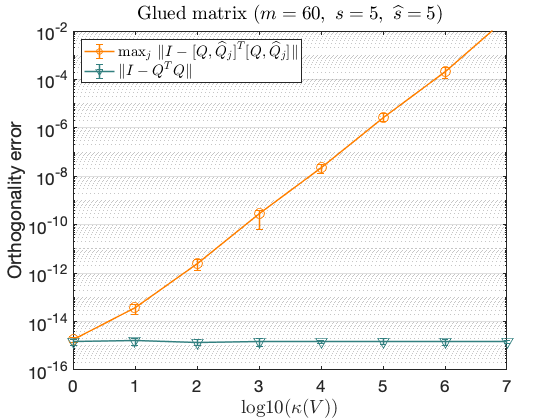}
    }
    \caption{Orthogonality errors.}
  \end{subfigure}
}
 \caption{Condition number and orthogonality error with one-step BCGS-PIP2 on {\tt glued} matrix.}\label{fig:error-pip2}
\end{figure}
% ============================== %

Fig.~\ref{fig:cholqr-error} shows the condition numbers and orthogonality errors when CholQR2 is used to orthogonalize the $10^5$-by-$5$ panel $\widehat{V}_j$ of varying condition numbers (i.e., $\widehat{V}_j := X \Sigma Y^T$ with random orthonormal $X$ and $Y$, and diagonal matrix $\Sigma$ with logspace singular values). It shows that as indicated by the bound~\eqref{eq:cholqr_bound}, the orthogonality error of $\widehat{Q}_j$ after the first CholQR grows as $\kappa(\widehat{V}_j)^2\mathcal{O}(\epsilon)$. Hence, when $\kappa(\widehat{V}_j) < \mathcal{O}(\epsilon)^{1/2}$, the condition number of $\widehat{Q}_j$ stays $\mathcal{O}(1)$, and we obtain the $\mathcal{O}(\epsilon)$ orthogonality error of $Q_j$ as indicated by Theorem~\ref{thm:chol}.

% ============================== %
\begin{figure}[t]
  \begin{subfigure}[b]{\linewidth}
   \centerline{
     \includegraphics[width=.7\linewidth]{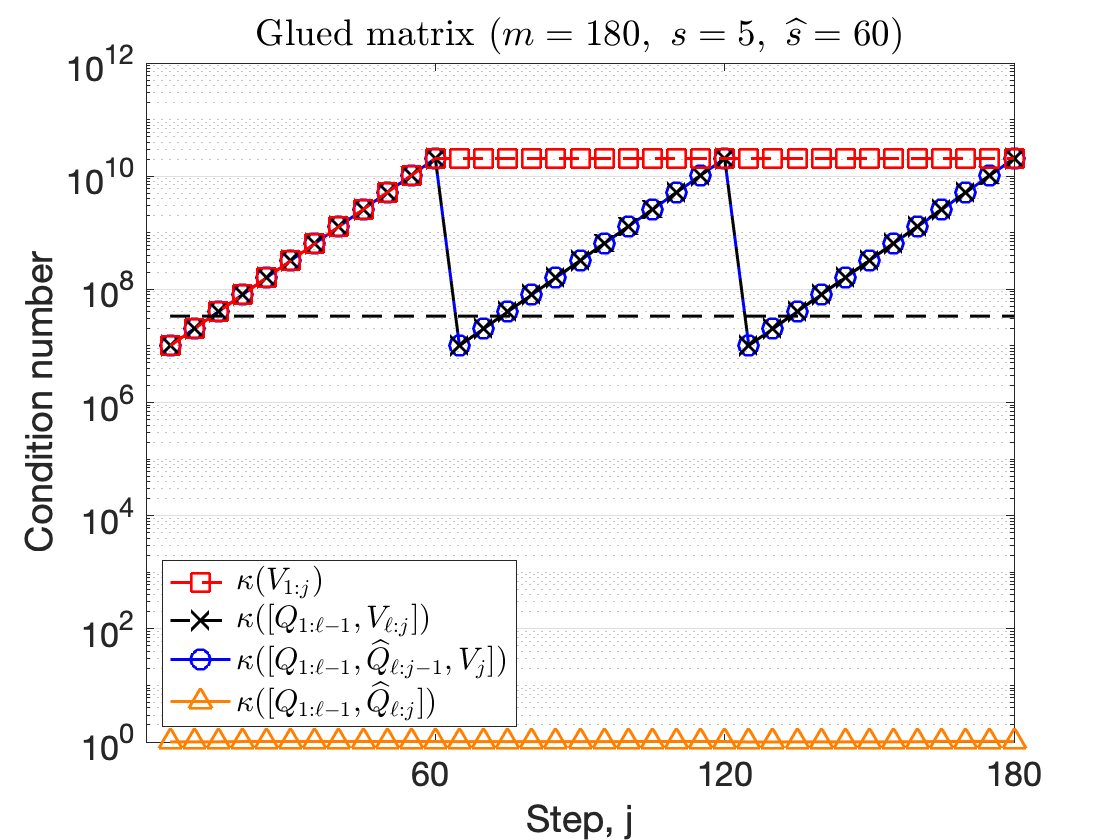}
    }
    \caption{Condition number (marker at every $s$ steps).}\label{fig:cond-2stage}
  \end{subfigure}
  \begin{subfigure}[b]{\linewidth}
   \centerline{
     \includegraphics[width=.7\linewidth]{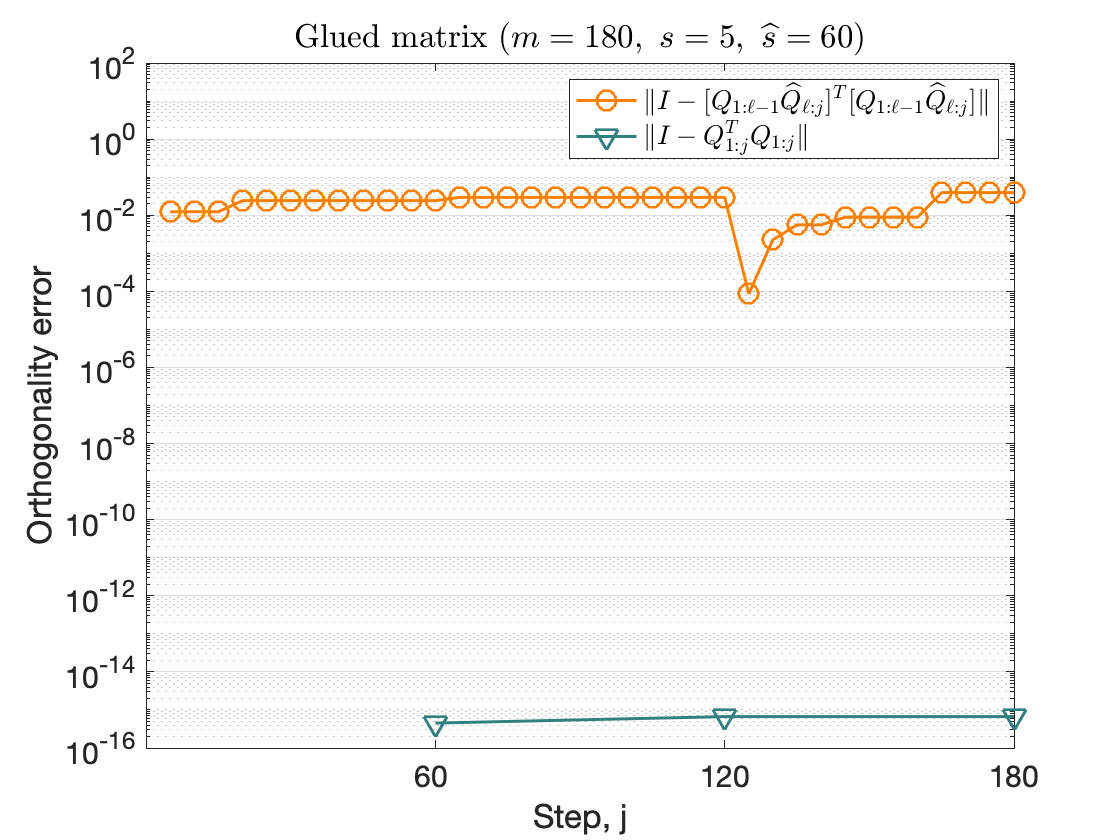}
    }
    \caption{Orthogonality errors (orange circle marker at every $s$ steps, while green triangle marker at every $\widehat{s}$ steps).}
  \end{subfigure}
 \caption{Condition number and orthogonality error using two-stage approach on {\tt glued} matrix with $(n,m,\hat{s},s)=(100000, 180, 60, 5)$.} \label{fig:error-2stage}
\end{figure}
% ============================== %

% ============================== %
\begin{figure}[t]
  \begin{subfigure}[b]{\linewidth}
   \centerline{
     \includegraphics[width=.85\linewidth]{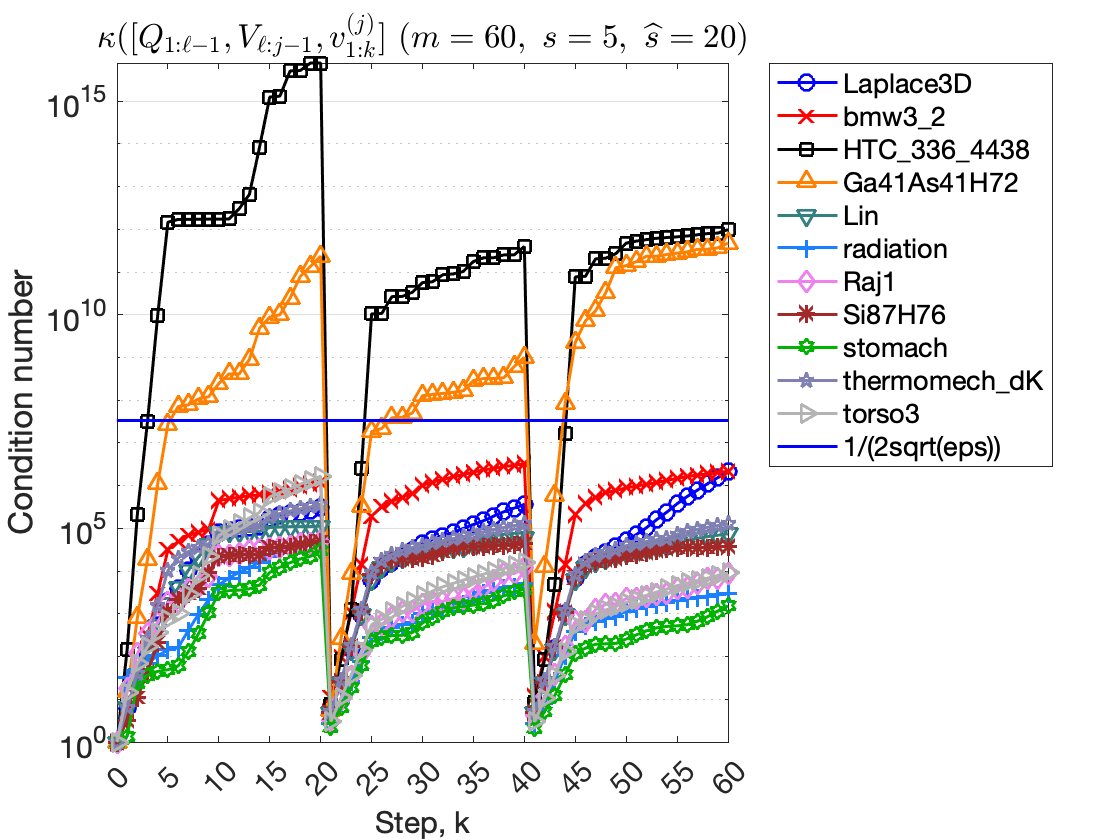}
    }
    \caption{Condition number of $[Q_{1:\ell-1},V_{\ell:j-1},v^{(j)}_{1:k}]$ (marker at every step).}
  \end{subfigure}
  \begin{subfigure}[b]{\linewidth}
   \centerline{
     \includegraphics[width=.85\linewidth]{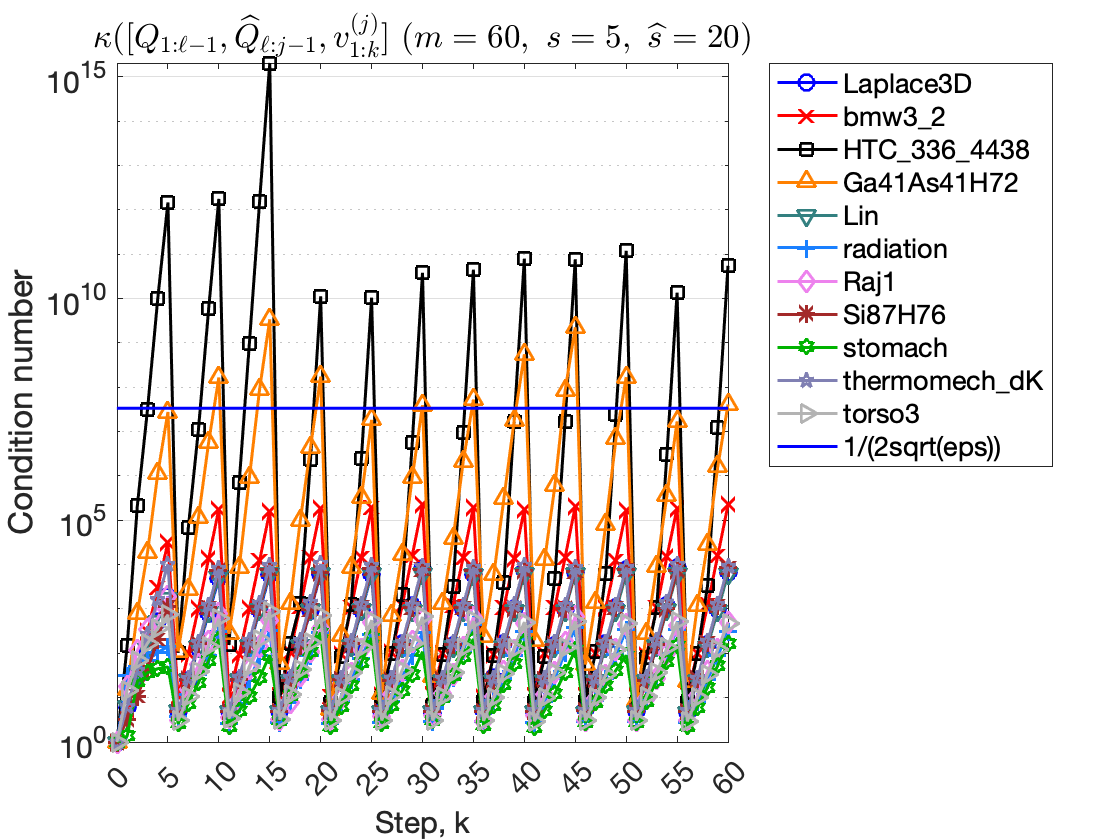}
    }
    \caption{\mbox{Condition number of $[Q_{1:\ell-1},\widehat{Q}_{\ell:j-1},v^{(j)}_{1:k}]$ (marker at every step).}}
  \end{subfigure}
  \begin{subfigure}[b]{\linewidth}
   \centerline{
     \includegraphics[width=.7\linewidth]{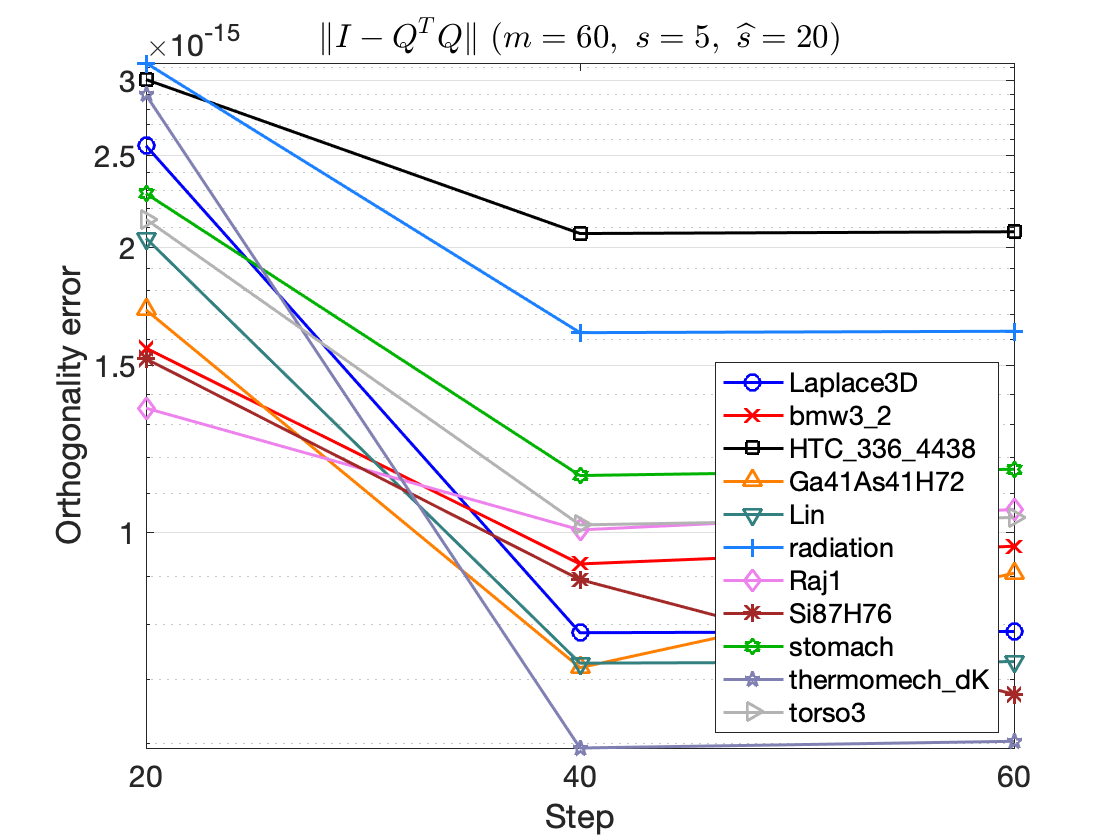}
    }
    \caption{Orthogonality error of $Q$ (marker at every $\widehat{s}$ steps).}
  \end{subfigure}
 \caption{Condition number and orthogonality error with Krylov vectors generated by MPK.} \label{fig:cond-mpk}
 \end{figure}
% ============================== %

Fig.~\ref{fig:error-pip2} shows the condition number and the orthogonality error when BCGS-PIP2 is used to orthogonalize the {\tt glued} matrix that has the same specified order of the condition number for each panel and for the overall matrix. As expected, when the condition number of the input matrix is smaller than $\mathcal{O}(\epsilon)^{-1/2}$, the orthogonality error of the basis vectors $\widehat{Q}$ after the first BCGS-PIP is bounded by $\kappa(V)^2 \mathcal{O}(\epsilon)$, and as a result, their condition number remained to be $\mathcal{O}(1)$. Consequently, after the second BCGS-PIP, the orthogonality error of the basis vector $Q$ was $\mathcal{O}(\epsilon)$,
which was the same error obtained by BCGS2 with CholQR2. 

Fig.~\ref{fig:error-2stage} shows the orthogonality errors using the two-stage approach. The test matrix is the {\tt glued} matrix, where each panel $V_j$ has the condition number $\mathcal{O}(10^7)$ but the condition number of $V_{1:j}$ grows as $2^{j-1}\mathcal{O}(10^7)$.
For this synthetic matrix with the pre-generated panels, after the first stage, the accumulated condition number of the panels $[Q_{1:\ell-1},V_{\ell:j}]$ was still about the same as the condition number of the original big panel $V_{\ell:j}$. Even though this synthetic matrix breaks the required condition~\eqref{eq:assumption-3}, the pre-processing step managed to keep the $\mathcal{O}(1)$ condition number of the big panel $[Q_{1:\ell-1},\widehat{Q}_{\ell:t}]$, and the overall orthogonality error of $Q$ was $\mathcal{O}(\epsilon)$.

Finally, Fig.~\ref{fig:cond-mpk} shows the condition number of the basis vectors that are generated by MPK combined with the two-stage block orthogonalization scheme. 
Unlike the pre-generated panels of the synthetic matrix in Fig.~\ref{fig:error-2stage},
the $s$-step basis vectors of the big panel $V_{\ell:t}$ are now generated by MPK, being interleaved with the pre-processing by BCGS-PIP.
As a result, unlike what we have observed in Fig.~\ref{fig:cond-2stage},
the accumulated condition number of $[Q_{1:\ell-1}, V_{\ell:j}]$ in Fig.~\ref{fig:cond-mpk} did not increase significantly as more panels were appended, and except for the two matrices ``HTC\_336\_4438'' and ``Ga41As41H72'',
the condition number of the big panel satisfied the required condition~\eqref{eq:assumption-3}.
The condition number of the basis vectors are managed likely because BCGS-PIP now pre-processes the basis vectors $\widehat{Q}_{j-1}$ before MPK generates the next set of the $s$-step basis vectors $V_j$ such that the starting vector $v^{(j)}_{1}$ is roughly orthogonal to the space spanned by the previous panels $\underline{V}_{\ell:j-1}$.
Without pre-processing the basis vectors, the condition number will continue to increase, preventing us from using a large step size.
Overall, after the second BCGS-PIP on big panel, the orthogonality errors of $Q$ was $\mathcal{O}(\epsilon)$ for all the matrices tested.

%\ichi{add $s$-step convergence plot?}

\section{Implementation}

We have implemented all the block orthogonalization algorithms for $s$-step GMRES within the Trilinos software framework \cite{Trilinos:2005,trilinos-website}.
Trilinos is a collection of open-source software libraries, called packages, 
for solving linear, non-linear, optimization, and uncertainty quantification problems.
It is intended to be used as building blocks for developing large-scale scientific or engineering applications.
Hence, any improvement in the solver performance could have direct impacts to the application performance.
In addition, Trilinos software stack provides portable performance of the solver on different hardware architectures,
with a single code base.
In particular, our implementation is based on Tpetra
%\footnote{The Tpetra Project Website:\url{https://trilinos.github.io/tpetra.html}} 
for distributed matrix and vector operations and Kokkos-Kernels
%\footnote{Kokkos Kernels Github: \url{https://github.com/kokkos/kokkos-kernels}}
for the on-node portable matrix and vector operations
(which also provides the interfaces for the vendor-optimized
 kernels like NVIDIA cuBLAS, cuSparse, and cuSolver).

On a GPU cluster, our GMRES uses GPUs to generate the orthonormal basis vectors,
where the matrices and vectors are distributed among MPI processes in 1D block row format
(e.g., using a graph partitoner like ParMETIS).
The operations with the small projected matrices, including
solving a small least-squares problem, is redundantly done on CPU by each MPI process.

Our focus is on the block orthogonalization of the vectors,
which are distributed in 1D block row format among the MPI processes.
The orthogonalization process mainly consists of
dot-products, vector updates, and vector scaling 
(e.g., $R_{1:j-1,j} := Q_{1:j-1}^TV_j$ and $V_j := V_j - Q_{1:j-1}R_{1:j-1,j}$ of BCGS in Figure~\ref{algo:bcgs}, 
   and $Q_j := V_j R_{j,j}^{-1}$ of CholQR in Figure~\ref{algo:cholqr}, respectively). 
The dot-products $Q_{1:j-1}^TV_j$ requires the global reduce among all the MPI processes,
and the resulting matrix $R_{1:j-1, j}$ is stored redundantly on all the MPI processes.
Given the upper-triangular matrix, the vectors can be updated and scaled locally without any additional 
communication.  All the local computations are performed by optimized kernels through Kokkos Kernels.

\section{Performance Results}
\label{sec:performance}

%\begin{figure}
%\centerline{
%  \includegraphics[width=\linewidth]{./figures/time-to-sol-summit.png}
%}
%  \caption{Time-to-solution on 2D Laplace on Summit, using $n_x = 2000$, $m=60$, $s=5$, and tol$=10^{-6}$.}
%\end{figure}

% \begin{table}
% \begin{center}\footnotesize
% \begin{tabular}{l|rrrrrr}
%  \# nodes  & 1 & 2 & 4 & 8 & 16 & 32\\
%  \hline\hline
%  %SpMV     & 68.0  & 35.2 & 25.2 & 19.5 & 16.6 & 15.5\\
%  Ortho    & 45.3  & 22.9 & 14.5 & 11.7 & 9.8  & 9.1\\  
%  %\hline
%  Total    & 107.2 & 58.0 & 39.9 & 31.4 & 26.5 & 24.5\\
% \end{tabular}
% \end{center}
%  \caption{
%           Time-to-solution for 2D Laplace, $n = 2000^2$, and the two-stage approach with $\widehat{s}=30$.
%           $s$-step GMRES needed 60,270 iterations to converge for all the cases, and SpMV times were about the same
%           as shown in Table~\ref{tab:time-to-sol}.} \label{tab:time-to-sol-shat}
% \end{table}

 \begin{table}
 \begin{center}\footnotesize
 \begin{tabular}{l|rr|rrrrrrr}
           &       &           &\multicolumn{4}{|c}{$\widehat{s}$}\\
           & GMRES & $s$-step  &  5    \ignore{& 10}    & 20    \ignore{& 30}    & 40    \ignore{&  50} & 60\\
  \hline\hline
  \# iters & 60251 & 60255 & 60255 \ignore{& 60260} & 60260 \ignore{& 60270} & 60280 \ignore{& 60290} & 60300\\
  \hline
  SpMV     & 100.1 & 103.6 & 103.4 \ignore{& 102.9} & 103.7 \ignore{& 102.4} & 104.3 \ignore{& 103.6} & 103.8\\
  Ortho    & 150.4 & 128.6 & 102.8 \ignore{& 95.0}  &  96.9 \ignore{& 76.9}  & 75.2  \ignore{& 70.9}  &  61.1\\
  \hline
  Total    & 249.7 & 232.3 & 206.4 \ignore{& 198.2} & 201.3 \ignore{& 179.4} & 180.2 \ignore{& 175.7} & 165.7
 \end{tabular}
 \end{center}
  \caption{
           Time-to-solution for 2D Laplace, $n = 2000^2$ on $4$ NVIDIA V100 GPUs, 
           with the two-stage approach using different values of second step size $\widehat{s}$,
           while the first step size is fixed as $s=5$.
           The first two columns, ``GMRES'' and ``$s$-step'', show the time using the standard
           and $s$-step GMRES, respectively.}
           \label{tab:time-to-sol-shat}
 \vspace{-.3cm}
 \end{table}

% ======================= %

\begin{table*}
\scriptsize
 \begin{center}
 \begin{tabular}{c|rrrr|rrrr|rrrr|rrrr}
  & \multicolumn{4}{c}{GMRES + CGS2} & \multicolumn{4}{|c}{$s$-step + BCGS2-CholQR2} & \multicolumn{4}{|c}{$s$-step + BCGS-PIP2} & \multicolumn{4}{|c}{$s$-step + Two-stage($\widehat{s}=m$)}\\
  \# nodes & \# iters & SpMV & Ortho & Total & \# iters & SpMV & Ortho & Total & \# iters & SpMV & Ortho & Total & \# iters & SpMV & Ortho & Total \\
  \hline\hline
  1        & 60251    & 63.5 & 100.2 & 164.3 
           %& 60255    & 63.0 & 59.4  & 122.6 & 60255    & 62.8 & 61.1 & 124.3  & 60300    & 68.2 & 67.6  & 134.7\\
           %&          &      &       &        
           %&          && 1.7$\times$ & 1.3$\times$ 
           %&          && 1.4$\times$ & 1.2$\times$ 
           %&          && 1.5$\times$ & 1.2$\times$\\
           %&          &      &       &    
           & 60255    & 64.2 & 71.9  & 134.1  & 60255   & 66.2 & 54.5 & 117.8 & 60300 & 66.6 & 32.0 & 99.2\\
           &          &      &       &    
           &          && $1.4\times$ & $1.2\times$ 
           &          && $1.8\times$ & $1.4\times$
           &          && $3.1\times$ & $1.7\times$\\
  \hline
  2        & 60251    & 38.2 & 72.9  & 108.5 
           %& 60255    & 35.3 & 37.1  & 72.6  & 60255    & 35.8 & 38.6 & 73.5   & 60300    & 36.9 & 30.0  & 68.1\\  
           %&          &      &       &       
           %&          && 2.0$\times$ & 1.5$\times$ 
           %&          && 1.9$\times$ & 1.5$\times$ 
           %&          && 2.4$\times$ & 1.6$\times$\\
           %&          &      &       &    
           & 60255    & 35.2 & 43.9  & 78.9 & 60255 & 35.0 & 30.1 & 65.2 & 60300 & 35.7 & 18.8 & 54.7\\
           &          &      &       &    
           &          && $1.7\times$ & $1.4\times$
           &          && $2.4\times$ & $1.7\times$
           &          && $3.9\times$ & $2.0\times$\\       
  \hline
  4        & 60251    & 27.7 & 59.8  & 85.6  
           %& 60255    & 25.7 & 26.1  & 51.7  & 60255    & 25.4 & 26.3  & 51.8   & 60300    & 26.0 & 20.3  & 46.9\\
           %&          &      &       &       
           %&          && 2.3$\times$ & 1.7$\times$ 
           %&          && 2.3$\times$ & 1.7$\times$ 
           %&          && 2.9$\times$ & 1.8$\times$\\
           %&          &      &       &    
           & 60255 & 25.3 & 30.8  & 57.1 & 60255 & 25.2 & 19.9 & 45.4 & 60300 & 27.1 & 12.6 & 40.2\\
           &          &      &       &    
           &          && $1.9\times$ & $1.5\times$
           &          && $3.0\times$ & $1.9\times$ 
           &          && $4.7\times$ & $2.1\times$\\ 
  \hline
  8        & 60251    & 20.0 & 51.9  & 70.8  
           %& 60255    & 20.0 & 22.0  & 41.7  & 60255    & 20.1 & 22.6 & 42.3   & 60300    & 19.9 & 15.5  & 35.4\\
           %&          &      &       &       
           %&          && 2.4$\times$ & 1.7$\times$ 
           %&          && 2.3$\times$ & 1.7$\times$
           %&          && 3.3$\times$ & 2.0$\times$\\
           %&          &      &       &    
           & 60255 & 20.0 & 27.2  & 47.0 & 60255 & 20.1 & 16.4 & 36.3 & 60300 & 19.5 & 10.8 & 30.6\\ 
           &          &      &       &
           &          && $1.9\times$ & $1.7\times$ 
           &          && $3.2\times$ & $2.0\times$ 
           &          && $4.8\times$ & $2.3\times$\\ 
  \hline
  16       & 60251    & 17.1 & 48.0  & 64.3  
           %& 60255    & 17.0 & 18.2  & 35.4  & 60255    & 17.5 & 20.3 & 37.1   & 60300    & 17.4 & 12.2  & 29.9\\
           %&          &      &       &       
           %&          && 2.6$\times$ & 1.8$\times$ 
           %&          && 2.4$\times$ & 1.7$\times$ 
           %&          && 3.9$\times$ & 2.1$\times$\\
           %&          &      &       &    
           & 60255 & 16.7 & 22.8  & 40.2 & 60255 & 17.1 & 14.1 & 30.9 & 60300 & 16.8 & 9.3 & 26.1\\
           &          &      &       &
           &          && $2.1\times$ & $1.6\times$ 
           &          && $3.4\times$ & $2.1\times$ 
           &          && $5.2\times$ & $2.5\times$\\ 
  \hline
  32       & 60251    & 16.0 & 46.9  & 61.9  
           %& 60255    & 15.8 & 17.7  & 33.4  & 60255    & 15.9 & 17.9 & 33.8   & 60300    & 15.8 & 10.5  & 26.1\\
           %&          &      &       &       
           %&          && 2.6$\times$ & 1.9$\times$ 
           %&          && 2.6$\times$ & 1.8$\times$ 
           %&          && 4.5$\times$ & 2.4$\times$\\
           %&          &      &       &    
           & 60255 & 15.6 & 22.3  & 38.2 & 60255 & 15.6 & 12.6 & 28.1 & 60300 & 16.0 & 8.7 & 24.5\\
           &          &      &       &    
           &          && $2.1\times$ & $1.8\times$ 
           &          && $3.7\times$ & $2.2\times$ 
           &          && $5.4\times$ & $2.5\times$\\ 
 \end{tabular}
 \end{center}
  \caption{
           Parallel Strong Scaling of time-to-solution with 9-points 2D Laplace, $n = 2000^2$. On each node, we launched six MPI processes (one MPI per GPU),
           and hence used 192 GPUs on 32 nodes.
            The table also shows the speedup gained
           using $s$-step and two-stage over standard GMRES for orthogonalization and total solution time.} \label{tab:time-to-sol}
\vspace{-.3cm}
 \end{table*}
% ======================= %

We now study the impact of different block orthogonalization schemes on the performance of $s$-step GMRES.
We used the restart length of 60 (i.e., $m=60$), and
considered GMRES to have converged when the relative residual
norm is reduced by six orders of magnitude.
We generated the right-hand-side vector such that the solution is a vector of all ones.

As discussed before, the step size $s$ may need to be carefully chosen.
For example, in Fig.~\ref{fig:error-2stage}, our two-stage algorithm pre-processes the
basis vectors at every fifth step to keep the condition number of the generated basis vectors small, but without the pre-processing step, the condition number will continue to increase exponentially after the fifth step. 
In practice, it is often infeasible to tune the step size as the condition number of the matrix could change significantly during the simulation. Hence,
to avoid numerical instability of MPK, in practice, a conservative step size like $s=5$ is used as the default step size. 
Since we are interested in improving the performance of block orthogonalization 
while using the small step size to maintain the stability of MPK,
we use this default step size of $s=5$ for all the performance results shown in this section, and study the effects of the
two-stage algorithms on the performance of $s$-step GMRES.

Table~\ref{tab:time-to-sol-shat} shows the performance with the two-stage approach using different values of the second step size $\widehat{s}$. The performance tests were conducted on the Advanced System Technology Testbed named Vortex at the Sandia National Laboratories.
Each node of Vortex has dual IBM Power 9 CPUs and four NVIDIA V100 GPUs.
We compiled our code using \textrm{GCC} version 8.3 and \textrm{CUDA} version 11.0 compilers.
As expected, the two-stage approach obtained higher performance using a larger step size, and it obtains the best performance when $\widehat{s}=m$. For these experiments, the pre-processing stage allowed us to maintain the numerical stability of the block orthogonalization process.

We conducted the remaining of our performance tests on the Summit supercomputer at Oak Ridge National Laboratory. Each compute node of Summit
has two 21-core IBM Power~9 CPUs and six NVIDIA Volta V100 GPUs.
The code was compiled using \textrm{g++} compiler version 7.5 and NVIDIA CUDA 11.0,
and linked to 
the IBM Engineering and Scientific Subroutine Library (ESSL) version 6.3 and 
Spectrum MPI version~10.4.

Table~\ref{tab:time-to-sol} shows the time to solution of $s$-step GMRES for solving 2D Laplace problem
on a 5-point stencil (strong parallel-scaling), using different block orthogonalization schemes:
\begin{itemize}
    \item Compared to BCGS2 with CholQR2 that the original $s$-step GMRES uses, BCGS-PIP2 reduces the number of synchronizations from five to two at every $s$ steps, and lowers the computational cost of the intra-block factorization by a factor of $1.5\times$.
          The table shows that BCGS-PIP improved the performance, especially as the latency starts to become more significant on a larger number of nodes. Specifically, BCGS-PIP reduced the orthogonalization time by a factor of $1.3\times$ and $1.7\times$
          over the original $s$-step method on 1 and 32 nodes, respectively, while achieving the respective speedups of $1.1\times$ and $1.3\times$ for the time-to-solution.
          
    \item Two-stage approach further reduces the orthogonalization time,
          and with $\widehat{s}=m$, it obtained the speedups of $1.7\times \sim 1.4\times$ over BCGS-PIP2, and hence, the time-to-solution was also reduced by factors of about $1.2\times$.
\end{itemize}

% ======================================= %
\begin{figure}[t]
\centerline{
   \begin{subfigure}[b]{.53\linewidth}
   %\centerline{
   \includegraphics[width=\linewidth]{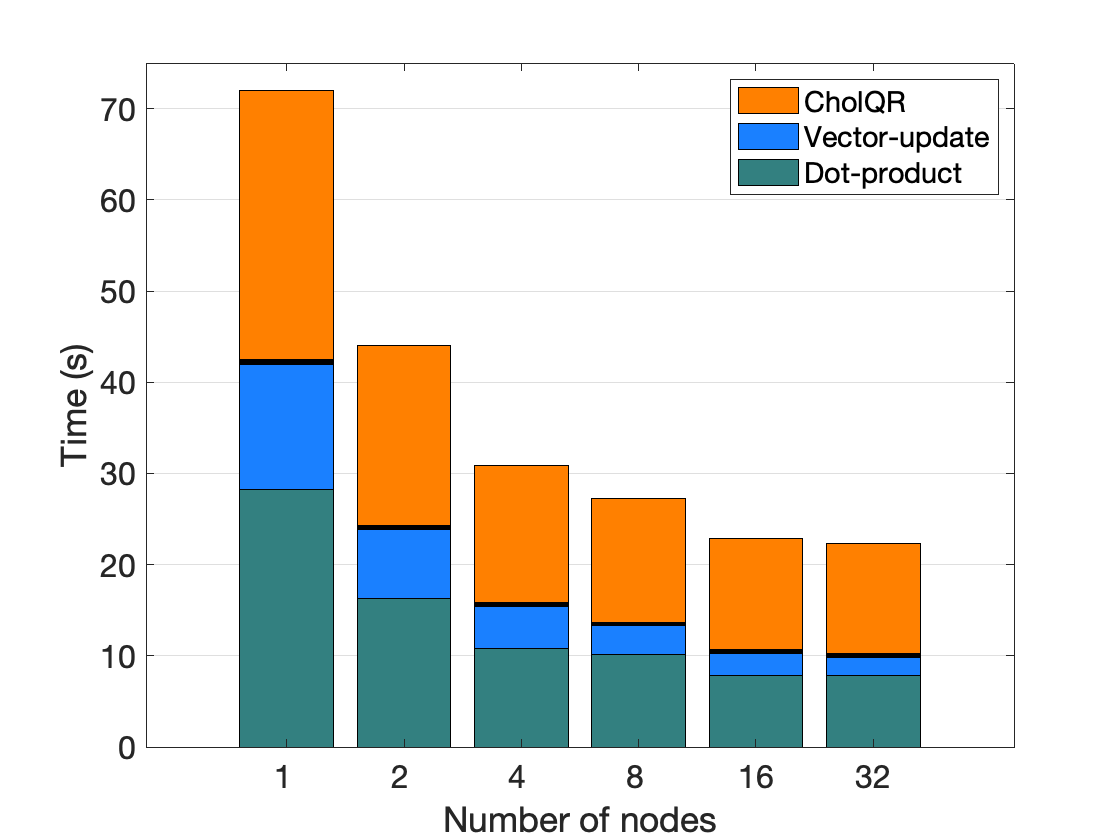}
   %}
   \caption{Time in seconds.}
   \end{subfigure}
   \hspace{-.5cm}
   \begin{subfigure}[b]{.53\linewidth}
   %\centerline{
   \includegraphics[width=\linewidth]{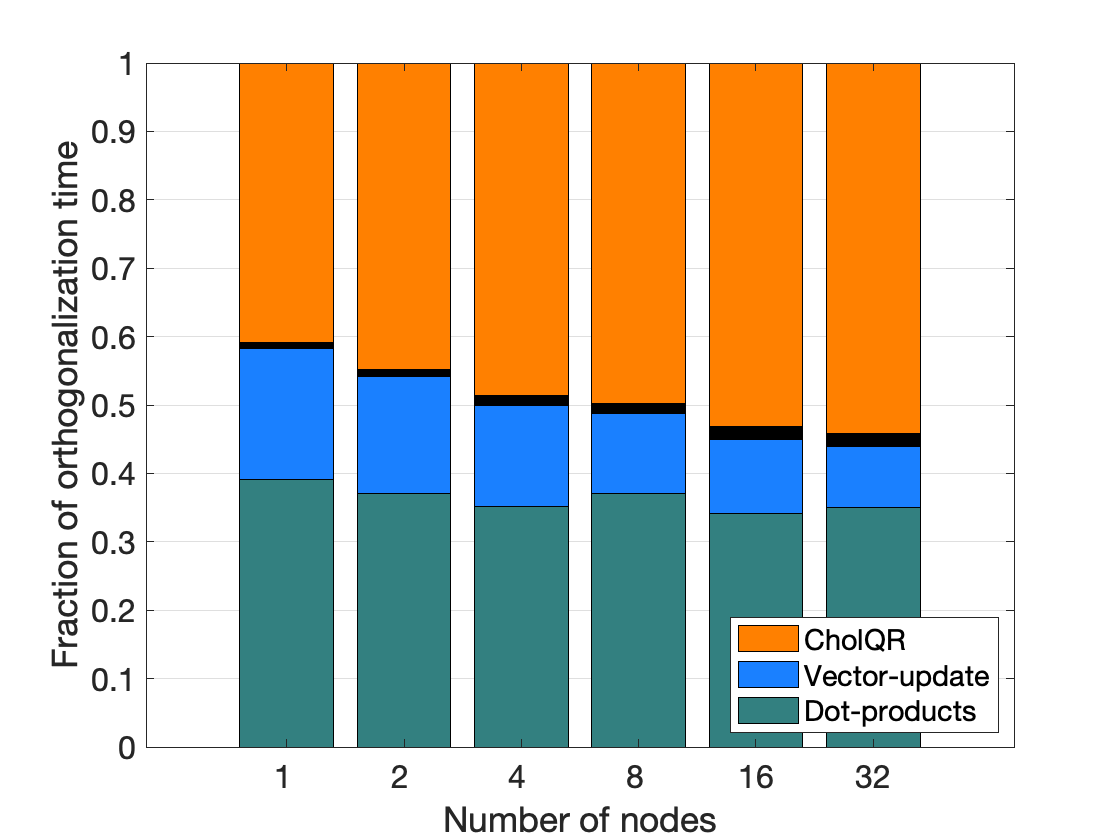}
   %}
   \caption{Fraction of time.}
   \end{subfigure}
}
   \caption{
   Orthogonalization time breakdown using BCGS2 with CholQR2 for 2D Laplace, $n = 2000^2$.} \label{fig:time-breakdown-bcgs2}
\vspace{-.3cm}
\end{figure}

\begin{figure}[t]
\centerline{
   \begin{subfigure}[b]{.53\linewidth}
   %\centerline{
   \includegraphics[width=\linewidth]{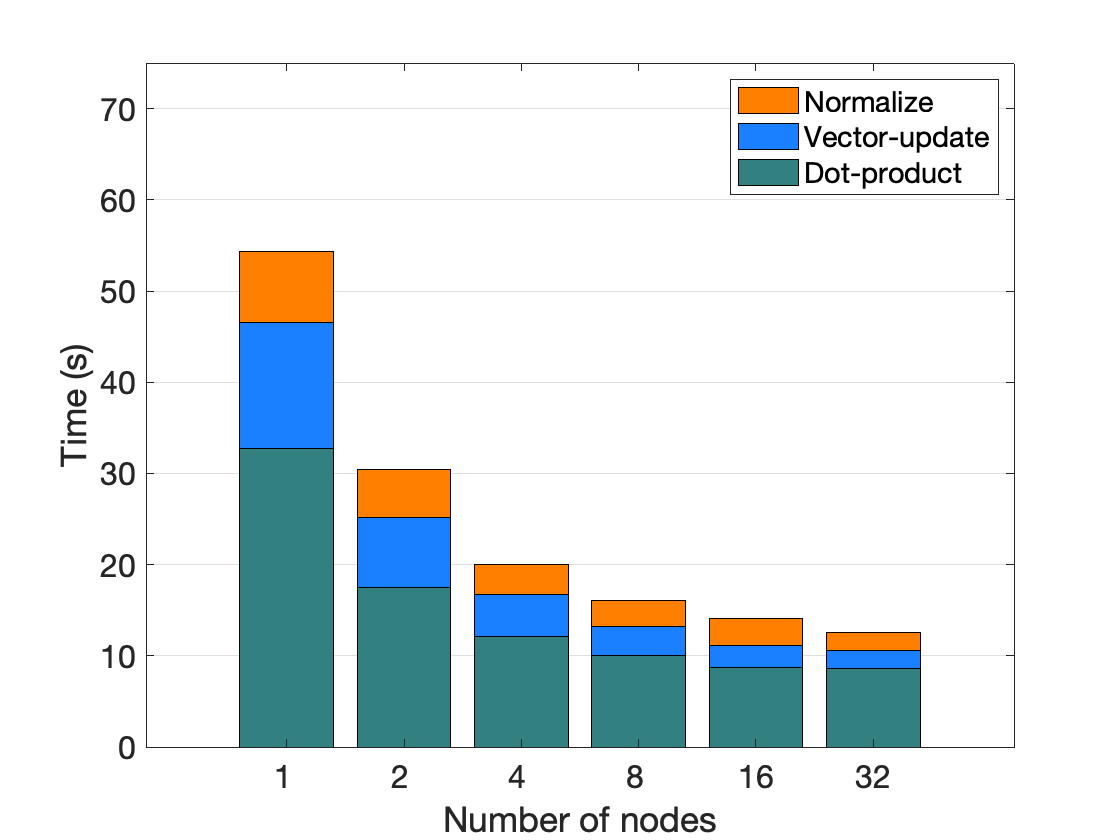}
   %}
   \caption{Time in seconds.}
   \end{subfigure}
   \hspace{-.5cm}
   \begin{subfigure}[b]{.53\linewidth}
   %\centerline{
   \includegraphics[width=\linewidth]{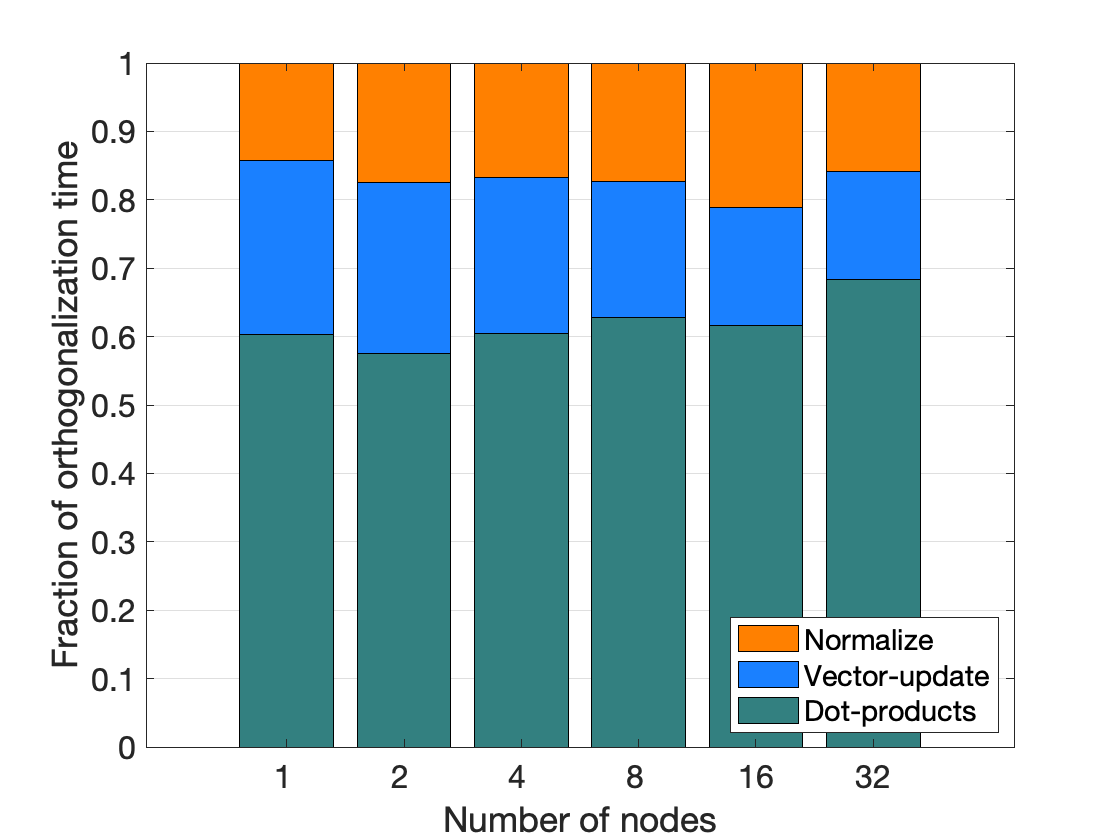}
   %}
   \caption{Fraction of time.}
   \end{subfigure}
}
   \caption{
   Orthogonalization time breakdown using BCGS-PIP2 for 2D Laplace, $n, = 2000^2$.} \label{fig:time-breakdown-pip2}
\vspace{-.3cm}
\end{figure}

\begin{figure}[t]
\centerline{
   \begin{subfigure}[b]{.53\linewidth}
   %\centerline{
   \includegraphics[width=\linewidth]{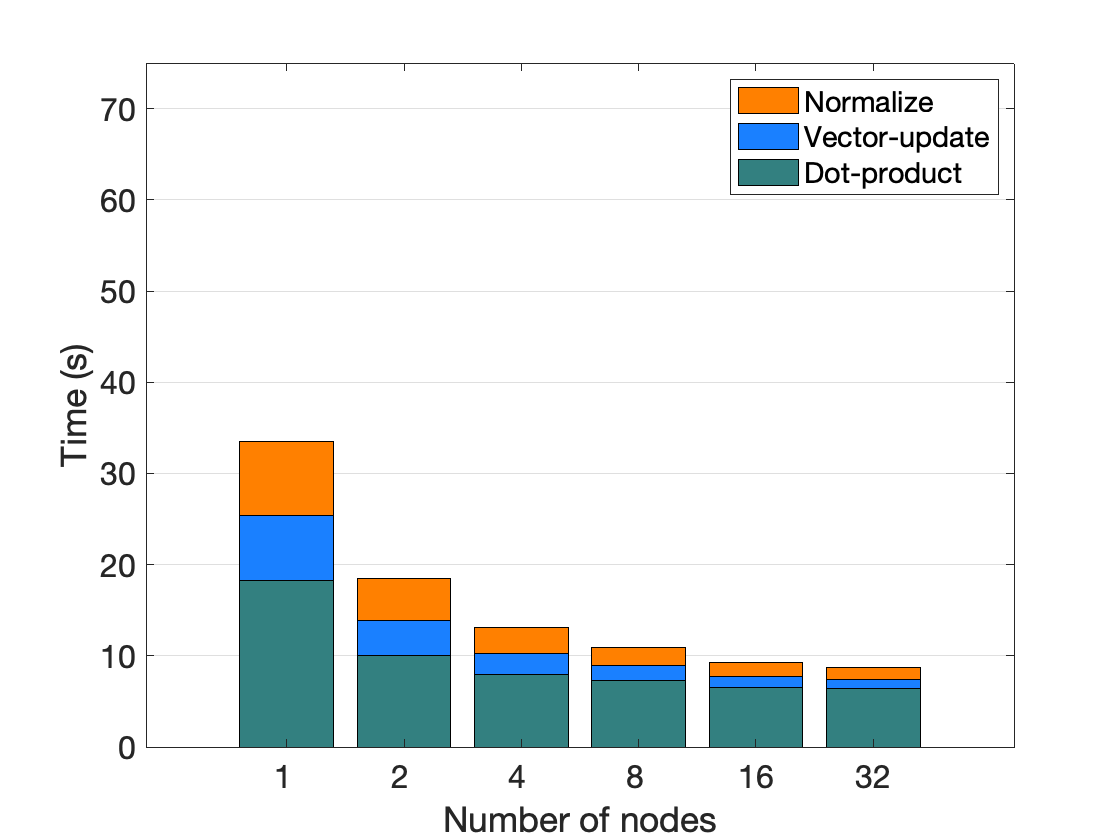}
   %}
   \caption{Time in seconds.}
   \end{subfigure}
   \hspace{-.5cm}
   \begin{subfigure}[b]{.53\linewidth}
   %\centerline{
   \includegraphics[width=\linewidth]{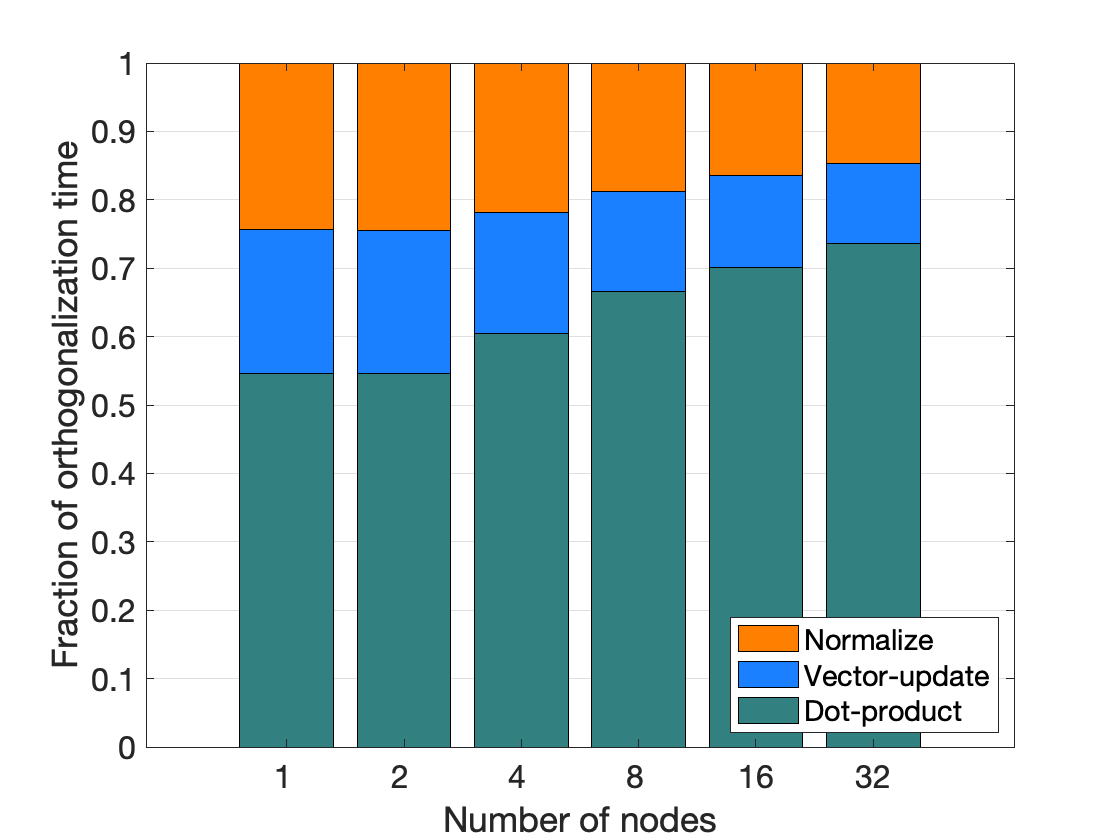}
   %}
   \caption{Fraction of time.}
   \end{subfigure}
}
   \caption{
   Orthogonalization time breakdown using two-stage approach for 2D Laplace, $(n,\widehat{s}) = (2000^2, m)$.} \label{fig:time-breakdown-2stage}
\end{figure}
% ======================================= %

Fig.~\ref{fig:time-breakdown-bcgs2} shows the breakdown of the orthogonalization time,
using BCGS2 with CholQR2. For BCGS2, we show the time needed to compute the dot-products
and vector-updates.
On a larger number of GPUs, the orthogonalization time becomes
dominated more by the dot-products with the global reduces, which are needed 
not only for BCGS2 but also for CholQR.
In comparison, 
Fig.~\ref{fig:time-breakdown-pip2} shows the breakdown of the orthogonalization time
using BCGS-PIP2
where the orthogonalization time was reduced by avoiding the global reduces and reducing the cost
of intra-block orthogonalization of CholQR.
Finally, Fig.~\ref{fig:time-breakdown-2stage} shows the breakdown of the orthogonalization time
using the two-stage approach with $\widehat{s}=m$.
The two-stage approach avoids these global reduces and further reduced the orthogonalization time.

 \begin{table}[t]\scriptsize
 \begin{center}
 \begin{tabular}{l|r|r|l|l}
          &          & \multicolumn{3}{|c}{Time / iter (ms)}\\
          & \# iters & SpMV & Ortho & Total\\
 \hline\hline
 \multicolumn{5}{l}{Laplace3D (Structured 3D model, SPD, $n=100^3$, $nnz/n=6.9$)}\\
 standard  & 454     & 0.36 & 0.87               & 1.15 \\
 $s$-step  & 455     & 0.38 & 0.43 ($2.0\times$) & 0.76 ($1.5\times$)\\
 bcgs-pip2 & 455     & 0.37 & 0.24 ($3.6\times$) & 0.60 ($1.9\times$)\\
 two-stage  & 480     & 0.37 & 0.16 ($5.4\times$) & 0.52 ($2.2\times$)\\
 \hline
 \multicolumn{5}{l}{Elasticity3D (Structured 3D model, SPD, $n=3 \cdot 100^3$, $nnz/n=5.7$)}\\
 standard  & 36      & 0.37 & 0.80               & 1.17\\
 $s$-step  & 40      & 0.39 & 0.45 ($1.8\times$) & 0.88 ($1.3\times$)\\
 bcgs-pip2 & 40      & 0.37 & 0.23 ($3.5\times$) & 0.65 ($1.8\times$)\\
 two-stage  & 60      & 0.33 & 0.14 ($5.7\times$) & 0.51 ($2.3\times$)\\
 \hline\hline
 \multicolumn{5}{l}{atmosmodl (CFD, numerically non-symmetric, $n=1.5$M, $nnz/n=6.9$)}\\
 standard  & 213     & 0.31 & 0.79               & 1.06\\
 $s$-step  & 215     & 0.37 & 0.38 ($2.1\times$) & 0.79 ($1.3\times$)\\
 bcgs-pip2 & 215     & 0.31 & 0.19 ($4.2\times$) & 0.50 ($2.1\times$)\\
 two-stage  & 240     & 0.35 & 0.14 ($5.6\times$) & 0.47 ($2.3\times$)\\
 \hline
\multicolumn{5}{l}{dielFilterV2real (Electromagnet, symmetric indefinite, $n=1.2$M, $nnz/n=41.9$)}\\
 standard  & 491856  & 0.36 & 0.99                & 1.22\\
 $s$-step  & 493145  & 0.33 & 0.36 ($2.8\times$) & 0.66 ($1.8\times$)\\
 bcgs-pip2 & 491865  & 0.30 & 0.19 ($5.2\times$) & 0.48 ($2.5\times$)\\
 two-stage & 491880  & 0.31 & 0.11 ($9.0\times$) & 0.42 ($2.9\times$)\\
 \hline
 \multicolumn{5}{l}{ecology2 (Circuit, SPD, $n=1.0$M, $nnz/n=5.0$)}\\
 standard  & 3471536 & 0.25 & 0.80                & 1.04\\
 $s$-step  & 3471540 & 0.24 & 0.34 ($2.4\times$) & 0.58 ($1.8\times$)\\
 bcgs-pip2 & 3471535 & 0.24 & 0.18 ($4.4\times$) & 0.42 ($2.5\times$)\\
 two-stage & 3471540 & 0.25 & 0.10 ($8.0\times$) & 0.36 ($2.9\times$)\\
 \hline
% \multicolumn{5}{l}{G3\_Circuit (Circuit, SPD, $n=1.5$M, $nnz/n=4.8$)}\\
% standard  & 2090    & 0.30 & 0.84                & 1.09\\
% $s$-step  & 2065    & 0.31 & 0.39 ($2.2\times$) & 0.68 (1.60$\times$)\\
% bcgs-pip2 & ---\\
% two-stage  & ---\\
% \hline
% \multicolumn{5}{l}{Serena, (Gas Reservoir, SPD, $n=1.4$M, $nnz/n=46.1$)}\\
% standard  & 104926  & 0.33 & 0.80                & 1.10\\
% $s$-step  & 109255  & 0.38 & 0.37 ($2.2\times$) & 0.73 ($1.5\times$)\\
% bcgs-pip2 & 108930  & 0.37 & 0.21 ($3.8\times$) & 0.57 ($1.9\times$)\\
% two-stage  & ---\\
 \hline
 \multicolumn{5}{l}{ML\_Geer, (Structural, numerically non-symmetric, $n=1.5$M, $nnz/n=73.7$)}\\
 standard  & 1596564 & 0.28 & 0.74                & 1.00\\
 $s$-step  & 1664400 & 0.29 & 0.37 ($2.0\times$)  & 0.65 ($1.5\times$)\\
 bcgs-pip2 & 1613060 & 0.28 & 0.20 ($3.7\times$)  & 0.47 ($2.1\times$) \\
 two-stage & 1517460 & 0.28 & 0.11 ($6.2\times$)  & 0.39 ($2.6\times$)\\ 
 \hline
 \multicolumn{5}{l}{thermal2 (Unstructured thermmal FEM, SPD, $n=1.2$M, $nnz/n=7.0$)}\\
 standard  & 139188  & 0.26 & 0.81               & 1.06\\
 $s$-step  & 139190  & 0.26 & 0.36 ($2.2\times$) & 0.61 ($1.7\times$)\\
 bcgs-pip2 & 139190  & 0.25 & 0.20 ($4.1\times$) & 0.44 ($2.4\times$)\\
 two-stage & 139200  & 0.27 & 0.13 ($6.2\times$) & 0.39 ($2.7\times$)\\
 %\hline
 %\multicolumn{5}{l}{cage14 ($n=1.5$M, $nnz/n=18.0$)}\\
 %standard & 12     & 0.03 & 0.01                & 0.04\\
 %$s$-step & \\
 %sketch   & \\
 \end{tabular}
 \end{center}
 \caption{Time per iteration for 3D model problems and matrices from SuiteSparse Matrix Collection on 16 Summit nodes;
          ParMETIS to distribute the matrix among 96 GPUs.}\label{tab:collection}
\vspace{-.3cm}
 \end{table}

 To summarize the performance studies,
 Table~\ref{tab:collection} compares the performance of $s$-step GMRES for 3D model problems
 and matrices from the SuiteSparse Matrix Collection.
 Since these matrices have similar dimensions,
 the required orthogonalization time and the speedups gained using $s$-step with respective orthogonalization algorithms were similar.
 Though the ratio of the orthogonalization time over the iteration time depends on the required time for SpMV with the matrices,
 BCGS-PIP reduced the orthogonalization and iteration time by factors of $1.8\sim2.0\times$ and $1.3\sim1.8\times$ over the original $s$-step GMRES,
 which had already obtained the respective speedups of $1.8\sim2.8\times$ and $1.3\sim1.8\times$ over the standard GMRES.
The two-stage approach further improved the performance
obtaining the respective speedups of
$1.4 \sim 1.8\times$ and $1.1 \sim 1.3\times$
for the orthogonalization and time-to-solution.

Finally, Figure~\ref{fig:gs} shows a similar performance trend when a local Gauss-Seidel preconditioner (block Jacobi with Gauss-Seidel in each block~\cite{Baker:2011}) was used. We used the multicolor Gauss-Seidel~\cite{Deveci:2016}
from Kokkos Kernels to get good performance on the GPU.

 \begin{figure}
   \centerline{
   \includegraphics[width=.8\linewidth]{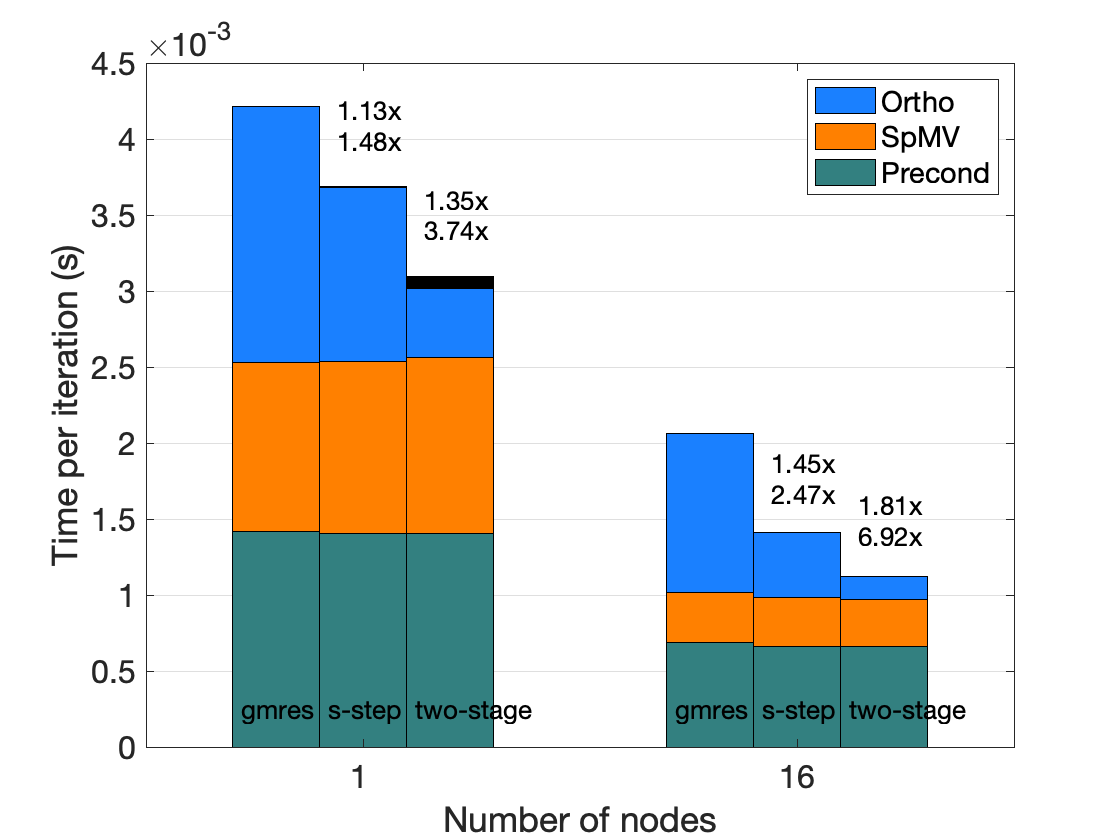}
   }

   \caption{
   Time per iteration breakdown of $s$-step GMRES with Gauss-Seidel preconditioner
   for 2D Laplace, $(n,\widehat{s}) = (2000^2, m)$, along with the speedups
   over standard GMRES for the orthogonalization (bottom) and iteration (top) time.} \label{fig:gs}
\vspace{-.3cm}
 \end{figure}

\section{Conclusion}

We surveyed the current state-of-the-art block orthogonalization algorithms for $s$-step GMRES,
and this motivated a new method called BCGS-PIP2. We showed BCGS-PIP2 reduces the cost of the orthogonalization and improves the performance of $s$-step GMRES. 
Nevertheless, since $s$-step basis vectors can be extremely ill-conditioned for a large step size $s$,
to maintain the stability in practice, a small step size needs to be used, which limits the performance
gain that $s$-step GMRES can bring.
In order to improve the performance of block orthogonalization using a small step size,
we introduced a two-stage algorithm, which pre-process the $s$ basis vectors at a time to maintain
the well-conditioning of the basis vectors but delay the orthogonalization until enough basis vectors
are generated to obtain higher performance.
We presented numerical and performance results to demonstrate its potential.

We are exploring the potential to combine this two-stage approach with other techniques, such as, random sketching.
This may allow us to remove many of the conditions required to guarantee the stability of the algorithm, without significant performance overhead.

%Finally, we believe our two-stage approach can be useful in other Krylov methods that rely on orthogonalization, but leave that as future work.

\section*{Acknowledgment}
%Removed for the double-blind review process.
This work was supported by the Exascale Computing Project (17-SC-20-SC), a collaborative effort of the U.S. Department of Energy Office of Science and the National Nuclear Security Administration.
Sandia National Laboratories is a multimission laboratory managed and operated by National Technology and Engineering Solutions of Sandia, LLC, a wholly owned subsidiary of Honeywell International, Inc., for the U.S. Department of Energy's National Nuclear Security Administration under contract DE-NA-0003525. This paper describes objective technical results and analysis. Any subjective views or opinions that might be expressed in the paper do not necessarily represent the views of the U.S. Department of Energy or the United States Government.

%\clearpage\newpage
\bibliographystyle{plain}
\bibliography{ref}

\begin{thebibliography}{10}

\bibitem{Bai:1994}
Z.~Bai, D.~Hu, and L.~Reichel.
\newblock {A {Newton} basis {GMRES} implementation}.
\newblock {\em IMA J. Numer. Anal.}, 14(4):563--581, 1994.

\bibitem{Baker:2011}
A.~H. Baker, R.~D. Falgout, T.~V. Kolev, and U.~M. Yang.
\newblock Multigrid smoothers for ultraparallel computing.
\newblock {\em SIAM J. Sci. Comput.}, 33:2864--2887, 2011.

\bibitem{Balabanov:2022}
O.~Balabanov.
\newblock Randomized {C}holesky {QR} factorizations, 2022.
\newblock arXiv:2210.09953.

\bibitem{Barlow:2021}
J.~L. Barlow.
\newblock Some added flexibility for block classical {Gram-Schmidt} with reorthogonalization, 2021.
\newblock Talk presented at the SIAM Conference on Applied Linear Algebra. Virtual.

\bibitem{Barlow:2013}
J.~L. Barlow and A.~Smoktunowicz.
\newblock Reorthogonalized block classical {G}ram–{S}chmidt.
\newblock {\em Numer. Math.}, 123:395--–423, 2013.

\bibitem{Carson:2015}
E.~Carson.
\newblock {\em Communication-avoiding {K}rylov subspace methods in theory and practice}.
\newblock PhD thesis, EECS Dept., U.C. Berkeley, 2015.

\bibitem{Carson:2022}
E.~Carson, K.~Lund, M.~Rozlo\v{z}ník, and S.~Thomas.
\newblock Block {G}ram-{S}chmidt algorithms and their stability properties.
\newblock {\em Linear Algebra Appl.}, 638:150--195, 2022.

\bibitem{Sturler:1995}
E.~{de Sturler} and H.~{van der Vorst}.
\newblock Reducing the effect of global communication in {GMRES}(m) and {CG} on parallel distributed memory computers.
\newblock {\em Applied Numer. Math.}, 18:441--459, 1995.

\bibitem{Demmel:2012}
J.~Demmel, L.~Grigori, M.~Hoemmen, and J.~Langou.
\newblock Communication-optimal parallel and sequential {QR} and {LU} factorizations.
\newblock {\em SIAM J. Sci. Comput.}, 34:A206--A239, 2012.

\bibitem{Deveci:2016}
M.~{Deveci}, E.~G. {Boman}, K.~D. {Devine}, and S.~{Rajamanickam}.
\newblock Parallel graph coloring for manycore architectures.
\newblock In {\em IEEE Int. Par. Dist. Proc. Symp. (IPDPS)}, pages 892--901, 2016.

\bibitem{Fukaya:2020}
T.~Fukaya, R.~Kannan, Y.~Nakatsukasa, Y.~Yamamoto, and Y.~Yanagisawa.
\newblock Shifted cholesky qr for computing the qr factorization of ill-conditioned matrices.
\newblock {\em SIAM J. Sci. Comput.}, 42:A477--A503, 2020.

\bibitem{Greenbaum:1997}
A.~Greenbaum, M.~Rozlo\v{z}n\.{i}k, and Z.~Strako\v{s}.
\newblock Numerical behaviour of the modified {G}ram-{S}chmidt {GMRES} implementation.
\newblock {\em {BIT} Numer. Math.}, 37:706--719, 1997.

\bibitem{Grigori:2015}
L.~Grigori and S.~Moufawad.
\newblock {Communication Avoiding ILU0 Preconditioner}.
\newblock {\em SIAM J. Sci. Comput.}, 37:C217--C246, 2015.

\bibitem{Trilinos:2005}
M.~A. Heroux, R.~A. Bartlett, V.~E. Howle, R.~J. Hoekstra, J.~J. Hu, T.~G. Kolda, R.~B. Lehoucq, K.~R. Long, R.~P. Pawlowski, E.~T. Phipps, A.~G. Salinger, H.~K. Thornquist, R.~S. Tuminaro, J.~M. Willenbring, A.~Williams, and K.~S. Stanley.
\newblock An overview of the {Trilinos} project.
\newblock {\em ACM Trans. Math. Softw.}, 31(3):397–423, sep 2005.

\bibitem{Hida:2001}
Y.~Hida, X.~Li, and D.~Bailey.
\newblock Algorithms for quad-double precision floating point arithmetic.
\newblock In {\em Proc. 15th IEEE Symp. Comput. Arith (ARITH-15)}, pages 155--162. IEEE, 2001.

\bibitem{Hoemmen:2010}
M.~Hoemmen.
\newblock {\em Communication-avoiding {Krylov} subspace methods}.
\newblock PhD thesis, EECS Dept., U.C. Berkeley, 2010.

\bibitem{Joubert:1992}
W.~Joubert and G.~F. Carey.
\newblock Parallelizable restarted iterative methods for nonsymmetric linear systems. {II}: parallel implementation.
\newblock {\em Int. J. Comput. Math.}, 44:269--290, 1992.

\bibitem{Mohiyuddin:2009}
M.~Mohiyuddin, M.~Hoemmen, J.~Demmel, and K.~Yelick.
\newblock Minimizing communication in sparse matrix solvers.
\newblock In {\em {Proc. Int. Conf. High Perf. Comput., Netw., Stor. and Anal. (SC)}}, pages 36:1--36:12, 2009.

\bibitem{Saad:1986}
Y.~Saad and M.~H. Schultz.
\newblock G{MRES}: a generalized minimal residual algorithm for solving nonsymmetric linear systems.
\newblock {\em SIAM J. Sci. Statist. Comput.}, 7:856--869, 1986.

\bibitem{Stath:2002}
A.~Stathopoulos and K.~Wu.
\newblock A block orthogonalization procedure with constant synchronization requirements.
\newblock {\em SIAM J.~Sci.~Comput.}, 23:2165--2182, 2002.

\bibitem{trilinos-website}
The {T}rilinos~{P}roject {T}eam.
\newblock {\em The {T}rilinos {P}roject {W}ebsite: \texttt{https://trilinos.github.io}}.

\bibitem{Weyl:1912}
H.~Weyl.
\newblock Das asymptotische {V}erteilungsgesetz der {E}igenwerte linearer partieller {D}ifferentialgleichungen (mit einer {A}nwendung auf die {T}heorie der {H}ohlraumstrahlung).
\newblock {\em Math. Annal.}, 71:441--479, 1912.

\bibitem{Fukaya:2015}
Y.~Yamamoto, Y.~Nakatsukasa, Y.~Yanagisawa, and T.~Fukaya.
\newblock Roundoff error analysis of the {C}holesky {QR}2 algorithm.
\newblock {\em Electronic Trans. Numer. Anal.}, 44:306--326, 2015.

\bibitem{Yamazaki:2014}
I.~Yamazaki, S.~Rajamanickam, E.~Boman, M.~Hoemmen, M.~Heroux, and S.~Tomov.
\newblock {Domain Decomposition Preconditioners for Communication-avoiding Krylov Methods on a Hybrid CPU-GPU Cluster}.
\newblock In {\em {Proc. Int. Conf. High Perf. Comput., Netw., Stor. and Anal. (SC)}}, pages 933--944, 2014.

\bibitem{Yamazaki:2020}
I.~Yamazaki, S.~J. Thomas, M.~Hoemmen, E.~G. Boman, K.~Swirydowicz, and J.~J. Elliott.
\newblock Low-synchronization orthogonalization schemes for \emph{s}-step and pipelined {K}rylov solvers in {T}rilinos.
\newblock In {\em Proc. of {SIAM} Conf. Parallel Processing for Sci. Comput.}, pages 118--128, 2020.

\bibitem{Yamazaki:2014:mpChol}
I.~Yamazaki, S.~Tomov, T.~Dong, and J.~Dongarra.
\newblock Mixed-precision orthogonalization scheme and adaptive step size for improving the stability and performance of {CA-GMRES} on gpus.
\newblock In {\em High Performance Computing for Computational Science - {VECPAR}}, volume 8969, pages 17--30, 2014.

\bibitem{Yamazaki:2015}
I.~Yamazaki, S.~Tomov, J.~Kurzak, J.~Dongarra, and J.~Barlow.
\newblock Mixed-precision block {Gram-Schmidt} orthogonalization.
\newblock In {\em Proc. 6th Workshop on Latest Advances in Scalable Algorithms for Large-Scale Systems}, pages 1--8, 2015.
\newblock Article No.~2.

\end{thebibliography}

\end{document}